\documentclass[a4paper, 11pt]{article}

\usepackage{amsmath,amsfonts,amssymb,mathrsfs,mathtools}
\usepackage{stmaryrd}
\usepackage{amsthm, textcomp}
\usepackage{verbatim}
\usepackage[frenchb,english]{babel}
\usepackage[latin1]{inputenc}
\usepackage[T1]{fontenc}
\usepackage{paralist}
\usepackage[page,title]{appendix}
\usepackage{xcolor}
\usepackage{graphics, array}
   \usepackage[pdftex]{graphicx}
   \usepackage{epstopdf}
\usepackage{float,cancel} 
\usepackage[normalem]{ulem}
\usepackage{listings, enumerate}
\usepackage{mathrsfs,dsfont}
\usepackage{vmargin}           
\usepackage{pstricks}
\usepackage{pstricks-add}
\usepackage{setspace}  
\usepackage{subfig, multirow}
\numberwithin{equation}{section}
\addtocounter{tocdepth}{1}
\begin{document}
 \title{Strong order of convergence of a semidiscrete scheme for the stochastic
Manakov equation}
\date{\today}
 \author{Maxime Gazeau}
  \maketitle

\newtheorem{theorem}{Theorem}[section]
\newtheorem{theoremf}{Théorème}[section]
\newtheorem{proposition}{Proposition}[section]
\newtheorem{property}{Property}[section]
\newtheorem{defn}{Definition}[section]
\newtheorem{lemma}{Lemma}[section]
\newtheorem{cor}{Corollary}[section]
\newtheorem{rmq}{Remark}[section]
\newtheorem{rmque}{Remarque}[section]
\newtheorem{ex}{Exemple}[section]
\newtheorem{condition}{Condition}[section]
\newtheorem{hyp}{Assumption}[section]
\newtheoremstyle{th}{}{}{}{}{}{}{\newline}
{{\color{black}\thmname{\bfseries #1}\thmnumber{{\bfseries \space #2}}}\thmnote{ (\rm #3)}}
\theoremstyle{th}
\newtheorem{etape}{Etape}[subsubsection]
\newtheorem{partie}{Part}[subsubsection]

\newcommand{\espace}[1]{\mathds{#1}}
\newcommand{\reps}[1]{#1^{R}_{\epsilon}} 
\newcommand{\mesure}[1]{\mathds{P}\circ \left(#1\right)^{-1}} 
\newcommand{\loi}{\xRightarrow[\epsilon \to 0]{}}

\renewcommand{\Re}{\mathcal{R}\mbox{e}}
\renewcommand{\Im}{\mathcal{I}\mbox{m}}
\newcommand{\abs}[1]{\left|#1\right|}
\newcommand{\norm}[2]{\left\|#1\right\|_{#2}}
\newcommand{\normm}[2]{ \left\|\left| #1 \right|\right\|_{#2}}
  \newcommand{\Id}{\text{I\footnotesize{d}}}

\newcommand{\puissance}[2]{#1^{#2}} 

\newcommand{\Reps}[1]{#1^{R}_{\epsilon}} 

\newcommand{\eps}[1]{#1_{\epsilon}} 
\newcommand{\vareps}[1]{#1_{\varepsilon}}
\newcommand{\mol}[1]{#1_{\eta}} 
\newcommand{\compoeps}[2]{#1_{#2,\epsilon}} 
\newcommand{\compovareps}[2]{#1_{#2,\varepsilon}} 
\newcommand{\rcompoeps}[2]{#1_{#2,\epsilon}^R} 
\newcommand{\compomol}[2]{#1_{#2,\eta}}
\newcommand{\sigb}{\pmb{\sigma}}
 \newcommand{\sigbeps}{\boldsymbol{\sigma}_{\epsilon}}

\newcommand{\dual}[2]{\ensuremath{\left< #1,#2 \right>}\xspace}
\newcommand{\ps}[2]{\ensuremath{\left( #1,#2 \right)_{\espace{L}^2}}\xspace}
\newcommand{\psr}[2]{\ensuremath{ #1 \cdot #2 }\xspace}
\newcommand{\psrdeux}[3]{\ensuremath{ #1.\left( #2, #3\right) }\xspace}
\newcommand{\pds}[3]{\frac{\partial^{#1} #3}{\partial #2^{#1}}}
\newcommand{\nl}[2]{F_{#1}{\left( #2\right) }}

\renewcommand{\a}{\begin{eqnarray*}} 
\renewcommand{\b}{\end{eqnarray*}}
\newcommand{\be}{\begin{eqnarray}} 
\newcommand{\ee}{\end{eqnarray}}
\newcommand{\bearray}{\begin{eqnarray}\left\{ \begin{array}{ll}}
\newcommand{\eearray}{\end{array}\right.\end{eqnarray}}
\newcommand{\bearrays}{\begin{eqnarray*}\left\{ \begin{array}{ll}}
\newcommand{\eearrays}{\end{array}\right.\end{eqnarray*}}

\newcommand{\saut}[1]{\vspace{#1\baselineskip}}
\newcommand{\leftsub}[2]{{\vphantom{#2}}_{#1}{#2}}
 
\renewcommand{\thefootnote}{(\arabic{footnote})} 
\renewcommand{\appendixpagename}{Annexes} 

\newcommand{\appsection}[1]{\let\oldthesection\thesection
  \renewcommand{\thesection}{Appendix \oldthesection.}
  \section{#1}\let\thesection\oldthesection}

\newcommand{\vecX}[2]{X^{#1}_{#2}}
\newcommand{\vecY}[2]{Y^{#1}_{#2}}
\newcommand{\vecXexact}[2]{\widetilde{X}^{#1}_{#2}}
\newcommand{\vecfX}[2]{\widehat{X}^{#1}_{#2}}
\newcommand{\vect}[3]{#1^{#2}_{#3}}
\newcommand{\vecinterp}[3]{\widetilde{#1}^{#2}_{#3}}
\newcommand{\vectilde}[3]{\widetilde{#1}^{#2}_{#3}}
\newcommand{\vecphi}[2]{\pmb{\phi}^{#1}_{#2}}
\newcommand{\conjvecX}[2]{\overline{X}^{#1}_{#2}}
\newcommand{\conjvecY}[2]{\overline{Y}^{#1}_{#2}}
\newcommand{\X}[2]{X^{#1}_{#2}}
\newcommand{\Y}[2]{Y^{#1}_{#2}}
\newcommand{\conjvec}[3]{\overline{#1}^{#2}_{#3}}
\newcommand{\dt}{\Delta t}
\newcommand{\dx}{\Delta x}
\newcommand{\dz}{\Delta z}
\renewcommand{\arraystretch}{1.5}


\begin{abstract}
It is well accepted by physicists that the Manakov PMD equation 
is a good model to describe the evolution of nonlinear electric 
fields in optical fibers with randomly varying birefringence. 
In the regime of the diffusion approximation theory, an effective 
asymptotic dynamics has recently been obtained to describe this evolution.
This equation is called  the stochastic Manakov equation.
  In this article, we propose a semidiscrete version of a Crank Nicolson 
  scheme for this limit equation and we analyze the strong error. 
Allowing sufficient regularity of the initial data, we prove that the
numerical scheme has strong order $1/2$.

\saut{0.5}


\saut{0.5}

\noindent
 \textbf{MSC2010 subject classifications} : 60H15, 35Q55, 60M15.

\end{abstract}

%

\section{Introduction}
The development of Internet and of the Web, in the second half of the $20^e$ 
century, has allowed for a rapid progress of optical communication systems. 
Today, engineers and physicists are trying to rise the bandwidths capacity of 
these communication systems as the Internet traffic has increased the last few years.
However, 
some dispersive effects limit the rate of transmission of information. The
Polarization
Mode Dispersion (PMD), appearing when the two components of the electric field
do not 
travel with the same characteristics, is one of the limiting factors of high bit
rate 
transmissions. The Manakov PMD equation was derived from the Maxwell equations
to study 
light propagation over long distances in such optical fibers \cite{[3]}. 
Due to the various length scales present in this problem, a small parameter 
$\epsilon$ appears in the rescaled equation. Using separation of scales
techniques, 
the author proved in \cite{[gazeau], [gazeauPHD]} that the asymptotic dynamics
is 
described by a stochastic perturbation, in the stratonovich sense, of the
Manakov equation. 
In this article, we consider a semidiscrete version of a Crank Nicolson scheme
for the 
stochastic Manakov equation. Our aim is to analyze the order of the error for
this scheme
and we prove that the strong order is $1/2$. 

\saut{0.5}

Numerical simulations are used in practice to solve complicated stochastic
differential 
equations and to lighten some hidden behaviours such as large deviations. In
optics, 
numerical simulations of the stochastic Manakov equation may help to understand
the impact
of the Polarization Mode Dispersion (PMD) on the pulse spreading \cite{[gazeaupmd]}. 
Depending on the problem, one may not be interested in the same quantities. On
one hand, 
one may be interested in the computation of path samples (related to strong
solutions) to 
emphasize, for example, the relation between various parameters in the
dynamics. 
On the other hand, if the quantity under interest depends only on the law of the
dynamics, 
one will focus on weak approximations. The pathwise error analysis of numerical
schemes for 
SDE has been intensively studied \cite{[Faure], [Kloeden], [Milsteinbook],
[TalayF]}, whereas 
the weak error analysis started later with the work of Milstein
\cite{[MilsteinI], [MilsteinII]} 
and Talay \cite{[Talay]}, who used the Kolmogorov equation associated to the SDE
to obtain a weak 
order of convergence. Usually, for Euler schemes, the strong order is $1/2$.
More sophisticated 
schemes exist to increase the pathwise order but their numerical implementation
requires to 
compute multiple iterated integrals, which may be difficult if the dynamics is
driven by a
multi-dimensional Brownian motion. 

The numerical analysis of SPDEs combines stochastic analysis together with PDEs 
numerical approximation. Most of the results are concerned with the analysis of 
pathwise convergence for solutions of semi-linear and quasi-linear parabolic 
equations (for a non exhaustive list, see 
\cite{[DavieGaines], [GyongyI], [GyongyMilletI], [GyongyNualart], [Hausenblas], 
[MilletMorien], [printems]}). There is some recent literature on dispersive
equations, 
both for stochastic nonlinear Schr\"odinger equations 
\cite{[bouardschema], [bouardorder], [MartyScheme]} and for a 
stochastic Korteweg-de-Vries equation \cite{[PrintemsDebusscheKdv], 
[PrintemsDebusscheCVKdv]}. Weak order for SPDEs has been considered 
later \cite{[DebusscheWeak], [PrintemsDebusscheWeak], [HausenblasWeak]}; 
the proof consists then in using the Kolmogorov equation which is now a 
PDE with an infinite number of variables.

In our case, the difficult and innovative point lies in the linear estimate. 
Indeed, the noise term contains a one order derivative and hence cannot be 
treated as a perturbation \cite{[gazeau], [gazeauPHD]}. Moreover an implicit
discretization of the noise has to be considered to build a conservative
scheme and the delicate point, in order to obtain the strong error, is to
deal with random matrices. Indeed, the linear system to be solved contains
random coefficients and the expression of the global error contains terms 
that are not martingales. Hence, the usual arguments consisting of applying 
the Burkholder-Davis-Gundy inequality to the stochastic integral cannot be 
applied straightforwardly. The probability order for the nonlinear scheme
is obtained using classical arguments \cite{[bouardorder], [printems]}. 
This notion is not usual in the context of numerical analysis of stochastic
equation. It is weaker than the strong order in time and is used here 
because of the nonlinear drift. 

 In this article, we consider the order of convergence of a semi-discrete
scheme. 
 For smooth initial data, it is probable that the error analysis of the fully 
 discrete scheme is not a problem and that the strong order in space is the same
 as in the deterministic case.
\subsection{Presentation of the numerical scheme}
The stochastic Manakov equation is given by
\begin{equation}\label{stochasticmanakov}
idX(t) +\left( \pds{2}{x}{X(t)} + F\left( X(t)\right)\right) dt  
+i\sqrt{\gamma}\sum_{k=1}^3 \sigma_k\pds{}{x}{X(t)}\circ dW_k(t)=0, 
\quad t \geqslant 0, x\in \espace{R}
\end{equation}
where the $\espace{C}^2$ vector of unknown $X= (X_1, X_2)$ is a random process 
on a probability space $(\Omega, \mathcal{F}, \espace{P})$, $\gamma$ 
is a small positive parameter given by the physics of the problem, 
$W=(W_1, W_2, W_3)$ is a $3$-dimensional Brownian motion and $\circ$ 
denotes the Stratonovich product.   The matrices $\sigma_1, \sigma_2, \sigma_3$ 
are the Pauli matrices 
\a
\sigma_1=\begin{pmatrix}
 0 & 1\\
 1 & 0\\
 \end{pmatrix},\;\;
 \sigma_2=\begin{pmatrix}
 0 & -i\\
 i & 0\\
 \end{pmatrix},\;\;
  \sigma_3=\begin{pmatrix}
 1 & 0\\
 0 & -1\\
 \end{pmatrix},
\b
 and the nonlinear term is given by $F\left( X(t)\right) = \abs{X}^2X(t)$. 
 The equivalent It\^o formulation is given by
 \begin{equation}\label{stochasticmanakovito}
 dX(t) = \left( C_{\gamma}\pds{2}{x}{X(t)} +i\nl{}{X}(t)\right)  dt  
 -\sqrt{\gamma}\sum_{k=1}^3 \sigma_k\pds{}{x}{X(t)} dW_k(t).
\end{equation}
where $C_{\gamma}=i+\frac{3\gamma}{2}$. In the deterministic case (i.e. when
$\gamma=0$),
when one considers the Manakov Equation, both the mass (equal to the
$\espace{L}^2$ norm) 
and the Hamiltonian $H$ given by
\[
 H(X)=\frac{1}{2}\int_{\espace{R}}\abs{\pds{}{x}{X}}^2dx
-\frac{1}{4}\int_{\espace{R}}\abs{X}^4dx
\]
 are conserved as time varies. This is not the case for the stochastic Manakov 
 equation that preserves only the mass, the Hamiltonian structure being 
 destroyed by the noise \cite{[gazeau], [gazeauPHD]}. Several numerical 
 approximations have been proposed to simulate the solution of the 
 deterministic equation, such as the Crank-Nicolson scheme \cite{[fortin]}, 
 the relaxation scheme \cite{[besse]} and Fourier split-step schemes 
 \cite{ [Taha], [Weideman]}. These schemes are known to be conservative
 for the $\espace{L}^2$ norm. The time centering method, used to discretize 
 the second order differential operator in the CN and relaxation schemes,
 allows them to be conservative for a discrete Hamiltonian. On the contrary,
 the splitting scheme fails in preserving exactly $H$. 

The question that needs to be addressed is the discretization of the noise
term. 
There are actually two different approaches based on the fact that, in the
continuous 
case, Equation \eqref{stochasticmanakov} and  Equation
\eqref{stochasticmanakovito} 
are equivalent. Hence, one may either propose a semi-implicit discretization of
the 
Stratonovich integral, using the midpoint rule, or an explicit discretization of
the
It\^o integral. However, in the discrete setting, the two formulations are not
equivalent.
Indeed, the discrete $\espace{L}^2$ norm is not preserved when considering an
Euler 
scheme based on the It\^o equation, while the semi-implicit discretization of
the 
Stratonovich integral allows preservation of the mass. Note that the
conservation of 
the discrete mass immediately leads to the unconditional $\espace{L}^2$
stability of 
the scheme. 

There is actually a more profound reason that keeps us from using a numerical
scheme 
based on the It\^o equation; this reason lies in the fact that the noise term
contains
a one order derivative. It is well known from the deterministic literature,
that 
explicit schemes for the advection equation require a stability criterion (CFL
condition)
to converge, while implicit schemes are stable. When considering the It\^o
approach, 
the discretization of the stochastic integral has to be explicit in order to be 
consistent with the equation, since an implicit discretization converges to the
backward
It\^o integral. Therefore, the It\^o approach leads to a CFL condition that
depends 
on Gaussian random variables. Since they are not bounded, this random stability
condition
may be very restrictive.

 We consider a semi-discrete Crank-Nicolson scheme given by
\begin{eqnarray}\label{CNApprox}\left\{ \begin{array}{ll}
  \vecX{n+1}{N}- \vecX{n}{N}   +  H_{\dt,n} \vecX{n+1/2}{N}  
  -i F\left(\vecX{n }{N},\vecX{n+1 }{N} \right)   \dt =0   \\[0.20cm]
 F\left(\vecX{n }{N},\vecX{n+1 }{N}\right) =\frac{1}{2} 
 \left(\abs{\vecX{n }{N}}^2 +\abs{\vecX{n+1 }{N}}^2 \right)\vecX{n+1/2}{N},  
   \end{array}\right.
\end{eqnarray}
where $\vecX{n+1/2}{N}=  \left( \vecX{n+1}{N} +\vecX{n}{N}\right)/2$, 
the time step is denoted $\dt$  and 
$\sqrt{\dt}\chi_k^n = W_k\left( (n+1)\dt\right)  - W_k\left( n\dt\right),
k=1,2,3  $
is the noise increment. The random matrix operator $H_{\dt,n}$ is defined by
\begin{align}\label{randomop}
 H_{\dt,n} = - i \dt I_2\partial_x^2 +\sqrt{\gamma  \dt}\sum_{k=1}^3\sigma_k 
\chi_k^n \partial_x.
\end{align}
with domain $\mathscr{D}(H_{\dt,n})= \espace{H}^2\left( \espace{R}\right)  
\subset \espace{L}^2\left( \espace{R}\right)$ independent of $n$, where 
$\espace{H}^2\left( \espace{R}\right)$ is the space of functions in
$\espace{L}^2$
such that their first two derivatives are in $\espace{L}^2$. The $2 \times 2$ 
identity matrix is denoted by $I_2$.

\saut{0.5}

This paper is organized as follows. In section \ref{Notation}, we introduce 
some notations and the main result of this article. Then, following the
approach 
of \cite{[gazeau], [gazeauPHD]} for the continuous equation, we construct a 
discrete random propagator associated to the linear equation. In section
\ref{LE}, 
we study the linear Euler scheme with semi-implicit discretization of the noise
and prove that the strong order is $1/2$. In section \ref{AS}, we give a result
on the strong order of convergence for a nonlinear equation with globally 
Lipschitz nonlinear terms. From this result and following the arguments of
\cite{[bouardorder]},
we obtain that the order of convergence in probability and the almost sure order
are $1/2$. 
This theoretical result is numerically recovered in section \ref{numsim} where
almost sure 
convergence curves are displayed. Finally some technical results are proved in
section \ref{app}.

 \subsection{Notation and main result}\label{Notation}
For all $p\geqslant 1$, we define 
$\espace{L}^p(\espace{R})= \left(L^p(\espace{R}; \espace{C}) \right)^2$ 
the Lebesgue spaces of functions with values in $\espace{C}^2$. 
Identifying $\espace{C}$ with $\espace{R}^2$, we define a scalar 
product on $\espace{L}^2\left( \espace{R}\right) $ by
\[
\left(u,v\right)_{\espace{L}^2}=\sum_{i=1}^2\Re\left\lbrace 
\int_{\espace{R}}u_i\overline{v_i}dx\right\rbrace.
\]
We denote by $\espace{H}^m\left( \espace{R}\right), m\in\espace{N} $ 
the space of functions in $\espace{L}^2$ such that their $m$ first derivatives
are 
in $\espace{L}^2$. We will also use $\espace{H}^{-m}$ the topological dual
space 
of $\espace{H}^m$ and denote $\dual{.}{.}$ the paring between $\espace{H}^m$
and 
$\espace{H}^{-m}$. The Fourier transform of a tempered distribution $v \in
\mathcal{S}'(\espace{R})$ 
is either denoted by $\widehat{v}$ or $\mathcal{F}v$.  If $s \in \espace{R}$
then 
$\espace{H}^s$ is the fractional Sobolev space of tempered distributions 
$v \in \mathcal{S}'(\espace{R})$ such  that 
$(1+\abs{\xi}^2)^{s/2}\widehat{v} \in \espace{L}^2$. 
Let $\left(E, \norm{.}{E} \right)$ and $\left(F, \norm{.}{F} \right)$ be two 
Banach spaces. 
We denote by $\mathcal{L}\left( E,F\right)$ the space of linear continuous 
functions from $E$ into $F$, endowed with its natural norm. If $I$ is an
interval 
of $\espace{R}$ and $1\leqslant p \leqslant +\infty$, then $L^p\left(I;E\right)
$ 
is the space of strongly Lebesgue measurable functions $f$ from $I$ into $E$
such 
that $t  \mapsto \norm{f(t)}{E}$ is in $L^p(I)$. The space 
$L^p\left(\Omega, E\right) $ is defined similarly where 
$\left(\Omega, \mathcal{F}, \espace{P}\right)$ is a probability space. 

\saut{0.5}

We now recall some results obtained in \cite{[gazeau], [gazeauPHD]} on the 
existence of a solution for the system \eqref{stochasticmanakov}. Let 
$\left(\Omega, \mathcal{F}, \espace{P}\right) $ be a probability space on which 
is defined a $3$-dimensional Brownian motion $ W(t)=\left( W_k(t)
\right)_{k=1,2,3}$.
We endow this space with the complete filtration $\mathcal{F}_t$ generated by
$W(t)$. 
The local existence result obtained for \eqref{stochasticmanakov} is stated
below.
\begin{theorem} 
Let $X_0=v \in \espace{H}^1(\espace{R})$ then there exists a maximal stopping
time 
$\tau^*(v,\omega)$ and a unique strong adapted solution $X$ (in the
probabilistic sense) 
to \eqref{stochasticmanakov}, such that $X \in C\left([0, \tau^*),
\espace{H}^1\left(\espace{R}\right) 
\right)$ $\espace{P}-a.s$. Furthermore the $\espace{L}^2$ norm is almost surely
preserved, 
i.e, $\forall t \in [0, \tau^*),
\norm{X(t)}{\espace{L}^2}=\norm{v}{\espace{L}^2}$ and 
the following alternative holds for the maximal existence time of the solution
: 
\[
\tau^*(v,\omega) = +\infty \ \text{ or } \ \limsup\limits_{t \nearrow
\tau^*(v,\omega)} 
\norm{X(t)}{\espace{H}^1}=+\infty.
\]
Moreover if the initial data $X_0$ belongs to $\espace{H}^m, m \geqslant 1$,
then the 
corresponding solution belongs to $\espace{H}^m$.
\end{theorem}
The noise ($\gamma \neq 0$ in Equation \eqref{stochasticmanakov}) destroys the
Hamiltonian
structure of the deterministic equation and it seems that no control on the
evolution of 
the $\espace{H}^1$ norm is available from the evolution of the energy. However,
the occurrence 
of blow up in this model remains an open question. We assume that the set 
$\left\lbrace \dt \right\rbrace = \left\lbrace \dt_n \right\rbrace_{n \in
\espace{N}}$ 
is a discrete sequence converging to $0$. We define a final time $T>0$ and an
interval 
$[0,T]$ on which we will consider the approximation of the solution of
\eqref{stochasticmanakov}.
Moreover $N_T= [T/\dt]$, the integer part of $T/\dt$. Similarly for any stopping
time $\tau$, $N_{\tau}=[\tau/\dt] $. Moreover we write $t_n = n\dt$ for any $n
\in [\![  0, N  ]\!]$
where $N$ is either $N_T$ or $N_{\tau}$ according to the situation. We denote
by 
$L^{\infty}\left(0, T;\espace{H}^m\right)$ the space of all bounded sequences
for 
$n = 0, \cdots, N_T$ with values in $\espace{H}^m$ endowed with the supremum
norm
\[
\norm{\vecX{}{N}}{L^{\infty}\left(0, T;\espace{H}^m\right)}= \underset{n\dt
\leqslant T}
{\sup_{n \in \espace{N}}}\norm{\vecX{n}{N}}{\espace{H}^m}.
\]
Moreover for a $n\times n$ matrix $A =  \left\lbrace a_{ij} \right\rbrace $, the
uniform 
norm is defined by 
\[
 \normm{A}{\infty}= \max_{1 \leqslant i \leqslant n} \sum_{j=1}^n \abs{a_{ij}}
\]
and the spectral norm of $A$ is defined by 
\[
 \normm{A}{2}= \sqrt{\rho\left( A^*A\right) }
\]
where $A^*$ is the adjoint matrix of $A$ and $\rho$ is the spectral radius.
Finally we denote $dW_0(u) =du$ and we introduce the notations for 
$j,k \in \llbracket 0,3 \rrbracket$
\begin{align*}
&W^{n,s}_j\left( f\right)  = \int_{t_n}^s f(u) dW_j(u)\\
&W^{n,n+1}_{j,k}\left( f\right) = \int_{t_n}^{t_{n+1}}  \int_{t_n}^s f(u)
dW_j(u) dW_k(s).
\end{align*}
We recall that the Pauli matrices have the following properties
\begin{property}\label{pauli_matrices}
Let $j,k,l \in \llbracket 0,3 \rrbracket$, then 
\begin{itemize}
\item  Commutation relations : $\left[\sigma_j, \sigma_k \right] 
= 2i\sum_{l=1}^3\varepsilon_{jkl}\sigma_l$.
\item  Anticommutation relations : $\sigma_j \sigma_k + \sigma_k \sigma_j 
=2\delta_{jk}\cdot I_2$
 and $\sigma_j = \sigma_j^*$,
\end{itemize}
where $\varepsilon_{jkl}=(j-k)(k-l)(l-j)/2$ is the Levi-Civita symbol.
\end{property}
  We denote by $\vecXexact{n}{} =  \vecX{}{ }\left(t_{n}\right)$ the solution
of 
  Equation \eqref{stochasticmanakov}, evaluated at the point $t_n$. Let us now
give 
  the main result of this paper stating that the approximation of Equation 
  \eqref{stochasticmanakov} by the scheme \eqref{CNApprox} has an order $1/2$ in
probability.
\begin{theorem} \label{errorNL}
Assume that $X_0 \in \espace{H}^6$, then for any stopping time $\tau < \tau^*
\wedge T$ 
almost surely we have
\[
 \lim_{C \to +\infty}\espace{P}\left( \max_{n = 0, \ldots, N_{\tau}}
\norm{\vecX{n}{} -
 \vecXexact{n}{}}{\espace{H}^1} \geqslant C \dt^{1/2} \right) =0,
\]
uniformly in $\dt$. Then we say, according to \cite{[printems]}, that the scheme
has an
order $1/2$ in probability. Moreover, for any   $\delta <\frac{1}{2}$, there
exists a 
random variable $K_{\delta}$ such that 
\[
 \max_{n = 0, \ldots, N_{\tau}} \norm{\vecX{n}{} -
\vecXexact{n}{}}{\espace{H}^1} \leqslant K_{\delta}(T,\omega) \dt^{\delta}.
\]
\end{theorem}
\section{The linear equation.}\label{LE}
In this section, we study the approximation of the solution of the linear
equation. In other 
words, we estimate the error between the solution of
\begin{equation}\label{linearmanakovlimite}
idX(t)+ \pds{2}{x}{X(t)}  dt+i\sqrt{\gamma}\sum_{k=1}^3
\sigma_k\pds{}{x}{X(t)}\circ dW_k(t)=0,
\quad t \geqslant 0, x \in \espace{R}
\end{equation}
and its approximation by  the semidiscrete mid-point scheme
\begin{equation}\label{ManakovPMDlimitelineairesemidiscret}
  \vecX{n+1}{N}- \vecX{n}{N}   +  H_{\dt,n} \vecX{n+1/2}{N}=0 ,
\end{equation}
where the expression of $H_{\dt,n} $ is given in \eqref{randomop}. The operator
$H_{\dt,n}$ 
is easily described thanks to the Fourier transform. Indeed, for any $\xi \in
\espace{R}$ 
\begin{equation}\label{symbol}
 \mathcal{F}\left( H_{\dt,n}\left( \xi\right)X \right) 
=\displaystyle{\begin{pmatrix}  
 i \dt \abs{\xi}^2  +i\sqrt{\gamma  \dt}\chi_3^n \xi & i\sqrt{\gamma  \dt}
 \left(\chi_1^n-i\chi_2^n \right) \xi\\[0.2cm]
     i\sqrt{\gamma  \dt}\left(\chi_1^n+i\chi_2^n \right) \xi          
     & i \dt \abs{\xi}^2 -i\sqrt{\gamma  \dt}\chi_3^n \xi
              \end{pmatrix}}\widehat{X}.
\end{equation}
Moreover, we set
\begin{equation}
 T_{\dt,n}=(\Id+\frac{1}{2}H_{\dt,n})\label{OpEx2},
\end{equation}
where $\Id$ is the identity mapping in $\espace{L}^2$.
To lighten the notation, we do not write the dependence in $N$ of the unknown 
$\vecX{ }{N}$. The aim of this section is to give an existence result of an
adapted 
solution for the scheme \eqref{ManakovPMDlimitelineairesemidiscret} and to give
an 
estimate of the discretization error. The results are stated in Propositions 
\ref{linearexistence} and \ref{linearstabilite} below.
\subsection{Existence and stability}
The next proposition states that the solution of the scheme 
\eqref{ManakovPMDlimitelineairesemidiscret} is uniquely defined and adapted, 
and that the mass is preserved.
\begin{proposition}\label{linearexistence}
Given $\vecX{ }{0} \in \espace{H}^m$ for $m \in \espace{N}$, there exists a 
unique adapted discrete solution $\left( \vecX{n}{} \right)_{n = 0, \cdots,
N_T}$ 
to \eqref{ManakovPMDlimitelineairesemidiscret} that belongs to $L^{\infty}(0, T;
\espace{H}^m)$.
Moreover the $\espace{H}^m$ norm of the solution $\vecX{n}{}$ of 
\eqref{ManakovPMDlimitelineairesemidiscret} is constant i.e. for all 
$n\in \llbracket 0, N_T \rrbracket$
\begin{equation}\label{normconstant}
\norm{\vecX{n}{}}{\espace{H}^m}=\norm{\vecX{}{0}}{\espace{H}^m}.
\end{equation}
\end{proposition}
\begin{proof}[Proof of Proposition \ref{linearexistence}]
Assume that $\vecX{n }{ }$ is a $\mathcal{F}_{n\dt}-$measurable random variable 
with values in $\espace{H}^m$. We set $A_{\dt}=\dt I_2\partial_x^2$ and 
$B_{\dt,n}=i\sqrt{\gamma\dt} \sum_{k=1}^3 \sigma_k \chi_k^n \partial_x $, 
for a.e. $\omega \in \Omega$. Using Property  \ref{pauli_matrices} of the 
Pauli matrices, Cauchy-Schwarz and Young inequalities, we may prove that a.s.
\begin{equation*}
\norm{B_{\dt,n}v}{\espace{L}^2}^2
 \leqslant \frac{1}{2}
\norm{A_{\dt,n}v}{\espace{L}^2}^2+\frac{C(\gamma,\omega)^2}{2}
\norm{v}{\espace{L}^2}^2,
\end{equation*}
where $C(\gamma,\omega) = 3\gamma \left(\chi_k^n(\omega) \right)^2$. 
Since $\abs{C(\gamma,\omega)} < +\infty \ \text{a.s.}$, we deduce thanks 
to the Kato-Rellich Theorem that $iH_{\dt,n}$ is selfadjoint in $\espace{L}^2$ 
with domain $\espace{H}^2$ and it follows that $T_{\dt,n}$ is invertible 
from $\espace{H}^2$ into $\espace{L}^2$. Hence, the unique
$\mathcal{F}_{t_{n+1}}$-- measurable 
solution is given by $\vecX{n+1}{}= U_{\dt,n}\vecX{n}{}$ a.s, where
\begin{equation}
 U_{\dt,n} =(\Id+\frac{1}{2}H_{\dt,n})^{-1}(\Id-\frac{1}{2}H_{\dt,n}).
\label{OpEx1} 
\end{equation}
 The conservation of the $\espace{L}^2$ norm follows because $H_{\dt,n}$ is
skew 
 symmetric and $ U_{\dt,n}^* U_{\dt,n}=\Id$.
\end{proof}
\begin{rmq}
 In our case, the operator $T_{\dt,n}$ is invertible for every $\dt$. Thus, 
 the implementation of the scheme  \eqref{ManakovPMDlimitelineairesemidiscret}
 does not require to use a truncation of the noise term as in \cite{[milstein]}
 to insure stability.
\end{rmq}  
\subsection{Strong order of convergence}
Let us now consider the order of convergence of the Crank Nicolson scheme 
\eqref{ManakovPMDlimitelineairesemidiscret}. To this purpose, we denote by 
$\vecXexact{n}{} =  \vecX{}{ }\left(t_{n}\right)$ the solution of 
\eqref{linearmanakovlimite}, evaluated at the point $t_n$, and define 
the vector error $e^{n}= \vecXexact{n}{}-\vecX{n}{}$. The error estimates
is given in the next result.
\begin{proposition}\label{linearstabilite}
If $\vecX{ }{0} \in \espace{H}^{m+5}$, $m \in \espace{N}$, then the scheme 
\eqref{ManakovPMDlimitelineairesemidiscret} is convergent and for any $p
\geqslant 1$ 
\begin{align}
\espace{E} \left(\max_{n \in \llbracket 0, N_T \rrbracket}
\norm{e^n}{\espace{H}^{m}}^{2p}\right)   \leqslant 
C(T, \gamma, p, \norm{\vecX{ }{0}}{\espace{H}^{m+5}} )  \dt^{p}.
\label{strongorder}
\end{align}
\end{proposition}
It may be surprising to require so much regularity on the initial data to prove 
a $L^p(\Omega)$ order for a linear equation. Usually, the order is obtained
using 
the explicit expression of the group $S(t)$, solution of the free Schr\"odinger 
equation (that is $\gamma=0$ in Equation \eqref{linearmanakovlimite}), and the
mild form of the It\^o equation. In our case, we cannot proceed similarly 
because of the semi-implicit discretization of the noise and the presence 
of a differential operator in this term.
\begin{proof}[Proof of Proposition \ref{linearstabilite}]
 Without loss of generality, we assume that $m=0$. The proof is divided into the
following steps.
\begin{enumerate} 
 \item Firstly, we evaluate the growth of the solution of the continuous
equation 
 \eqref{linearmanakovlimite}. More precisely, we denote by $\widetilde{e}^{ \,
n}(s)$ 
 the difference $\widetilde{e}^{ \, n}(s) = X(s) -  \vecXexact{n}{}$, for all 
 $s \in [t_n, t_{n+1}]$ and we give an estimate of it in the space 
 $L^{2p}\left(\Omega, L^{\infty}\left(0, T; \espace{L}^2\right)  \right) $.
\item Secondly, we write a discrete Duhamel equation for the global error 
$e^n = \vecX{n}{} -\vecXexact{n}{} $, where the Itô formulation of equation 
\eqref{linearmanakovlimite} is used.
\item The expression of the global error contains terms that are not martingales
and hence martingales inequalities cannot be applied straightforwardly. 
Therefore, we separate the adapted part to the non adapted one introducing 
a discrete random propagator $\mathcal{V}^{l}_{\dt}$. The adapted 
part is estimated thanks to the usual martingale inequalities, while a bound 
on the non-adapted part is obtained estimating the difference between 
$\mathcal{V}^{l}_{\dt}$ and the discrete random propagator appearing in 
the expression of the global error.
\end{enumerate}
\paragraph{Step 1.} The next lemma gives an estimate of the growth of the 
solution $\vecX{}{}(s)$ of \eqref{linearmanakovlimite} starting at
$\vecXexact{n}{}$. 
\begin{lemma}\label{strong_estimate_sol}
 For any $p \geqslant 1$, if $\vecX{}{0} \in \espace{H}^{1}$ then
\[ 
\espace{E}\left(\sup_{t_n \leqslant s \leqslant t_{n+1}}
\norm{\widetilde{e}^{ \, n}(s)}{\espace{L}^2}^{2p} \right) \leqslant 
C_p(\gamma)  \norm{X_0}{\espace{H}^{1}}^{2p} \dt^{p}   \qquad \forall n = 0,
\cdots, N_T -1.
\]
\end{lemma} 
\begin{proof}
Writing the It\^o formulation of Equation \eqref{linearmanakovlimite} under its
mild form, we get
\[
 X(t)- \vecXexact{n}{}= \left( S(t-t_n) - \Id\right) \vecXexact{n}{} + 
 i\sqrt{\gamma} \sum_{k=1}^3\int_{t_n}^t S(t-u)\sigma_k \partial_xX(u)  dW_k(u),
\]
where $S(t)$ is the semi-group solution of the linear equation 
$\partial_t X(t) = C_{\gamma} \partial_x^2 X(t)$ with $C_{\gamma} = 
i+ \frac{3\gamma}{2}$. Using the Fourier transform, it can easily be shown that
\[
\norm{\left(S(t)-\Id\right) f}{\espace{L}^2} \leqslant C(\gamma) 
t^{1/2}\norm{f}{\espace{H}^{1}}, \qquad  \forall f \in \espace{H}^{1},
\]
from which we deduce, together with \eqref{normconstant}, that 
\[
\espace{E}\left(\sup_{t_n \leqslant s \leqslant t_{n+1}} 
\norm{ \left( S(s-t_n)-\Id\right)   \vecXexact{n}{} }{\espace{L}^2}^{2p}
\right) 
\leqslant C_p(\gamma)\norm{\vecX{}{0}}{\espace{H}^{1}}^{2p}\dt^{p}.
\]
Moreover, since $X$ is adapted and belongs to $L^{2p}(\Omega, C([0,T],
\espace{L}^2))$, 
we may apply the Burkholder-Davis-Gundy inequality to the stochastic
convolution. 
Using the contraction property of the semigroup $S(t)$ and
\eqref{normconstant}, 
we obtain the estimate
\[\espace{E}\left(\sup_{t_n \leqslant s \leqslant t_{n+1}} \norm{ \sqrt{\gamma} 
\sum_{k=1}^3\int_{t_n}^s S(s-u)\sigma_k \partial_xX(u)  dW_k(u) 
}{\espace{L}^2}^{2p} \right) 
\leqslant C_p\norm{\vecX{}{0}}{\espace{H}^{1}}^{2p} \gamma^p \dt^p.
\]
This concludes the proof of the Lemma.
\end{proof}
\paragraph{Step 2.}
Using the It\^o formulation of Equation \eqref{linearmanakovlimite} and 
evaluating its solution on the time interval $[t_n, t_{n+1}]$, we obtain
\begin{align}
\vecXexact{n+1}{} &= \vecXexact{n}{}+ C_{\gamma}W_0^{n, n+1} \left(
\partial_x^2X \right) 
-\sqrt{\gamma}\sum_{k=1}^3\sigma_k W_k^{n, n+1}\left(
\partial_xX\right)\label{It\^oform}\\
&=  \vecXexact{n}{} - H_{\dt,n}\vecXexact{n+1/2}{} 
+  \epsilon_1^n +  \epsilon_2^n, \label{eqtruncationerror}
\end{align}
where the random variables  $\epsilon_1^n$ and $\epsilon_2^n$ are given by
\bearray\label{errorterms}
&\epsilon_1^n =iW_0^{n, n+1}\left( \partial_x^2X-
\partial_x^2\vecXexact{n+1/2}{}\right) \\
&\epsilon_2^n = \sqrt{\gamma}\sum_{k=1}^3\sigma_k
\left( \partial_x \vecXexact{n+1/2}{}W_k^{n}(1) - W_k^{n,n+1}\left(\partial_xX
\right) \right) 
+ \frac{3 \gamma}{2} W_0^{n,n+1}\left(\partial_x^2X \right). 
\eearray
By induction, we obtain the recursive formula for the global error
\[
 e^{n}=  \mathcal{U}^{n,0}_{\dt} e^0 +   \sum_{l=1}^n  \mathcal{U}^{n,l}_{\dt}
 \left( \epsilon_1^{l-1} +   \epsilon_2^{l-1}\right),
\]
where
\[
\mathcal{U}^{n,l}_{\dt}= \left\{ \begin{array}{lll} &U_{\dt,n-1}\cdots
U_{\dt,1}U_{\dt,0}  
\qquad &\text{for }   l = 0\\[0.25cm]
			&U_{\dt,n-1}\cdots U_{\dt,l}T_{\dt,l-1}^{-1} \qquad
&\text{for }  
			l \in \llbracket 1, n-1 \rrbracket\\[0.25cm]
                        &T_{\dt,n-1}^{-1} \qquad &\text{for } l=n.
                     \end{array}\right.
\]
Let us write the remainder term $\epsilon^{l-1}_1 $, given in
\eqref{errorterms}, 
as the sum of two terms $\epsilon^{l-1}_{1,1}$ and $\epsilon^{l-1}_{1,2}$.
Writing
\begin{equation}\label{exp_corr}
\vecXexact{l-1/2}{} = \vecXexact{l-1}{} +   \frac{1}{2} \left( \vecXexact{l }{} 
-\vecXexact{l-1}{} \right) 
\end{equation}
 and using Equation \eqref{It\^oform}, we obtain the following expressions for 
 $\epsilon_{1,1}^{l-1}$ and $ \epsilon_{1,2}^{l-1}$ 
\be\label{firstterm1}
\epsilon_{1,1}^{l-1} =  i C_{\gamma} \left( W_{0,0}^{l-1,l}\left( \partial^4_xX
\right) 
- \frac{1}{2}W_{0}^{l-1,l}\left( \partial^4_xX \right) \dt \right) 
\ee
and 
\be\label{firstterm2}
\epsilon_{1,2}^{l-1}&=& -     i\sqrt{\gamma}   
\sum_{k=1}^3 \left( W_{k,0}^{l-1,l}\left(\sigma_k \partial^3_xX \right) - 
\frac{1}{2}W_{k}^{l-1,l}\left( \sigma_k\partial^3_xX \right)\dt \right). 
\ee
 We proceed similarly for the term $\epsilon^{l-1}_2$ writing it as a sum of 
 three terms $\epsilon_2^{l-1} =  \epsilon^{l-1}_{2,1} +\epsilon^{l-1}_{2,2}+
 \epsilon^{l-1}_{2,3}$. Using again \eqref{exp_corr} and Equation
\eqref{It\^oform}, 
 the truncation error $\epsilon^{l-1}_{2}$, given in Expression
\eqref{errorterms}, 
 can now be expressed thanks to
\bearray\label{secondterm}
&\epsilon_{2,1}^{l-1} = - \sqrt{\gamma}\sum_{k=1}^3 W_k^{l-1,l}
\left(\sigma_k \partial_x \widetilde{e}^{\, l-1} \right)\\[0.2cm]
&\epsilon_{2,2}^{l-1} =\frac{3 \gamma}{2}W_0^{l-1,l}\left( \partial_x^2 X
\right) 
-\frac{ \gamma}{2}\sum_{j,k=1}^3\sigma_j\sigma_k W_j^{l-1,l}\left( \partial_x^2
X \right)
W_k^{l-1}(1) \\[0.2cm]
&\epsilon_{2,3}^{l-1} = \frac{\sqrt{\gamma} }{2} C_{\gamma}\sum_{k=1}^3
\sigma_k 
W_0^{l-1,l}\left( \partial_x^3 X \right) W_k^{l-1}(1).
\eearray
\paragraph{Step 3.} Since $\mathcal{U}^{n,l}_{\dt}$  depends on the Brownian 
increments after time $t_{l-1}$, it is not $\mathcal{F}_{t_{l-1}}$ adapted and
\[
 \mathcal{U}^{n,l}_{\dt}\sum_{k=1}^3 W_{k,0}^{l-1,l}\left(\sigma_k \partial^3_xX
\right)
 \neq \sum_{k=1}^3W_{k,0}^{l-1,l}\left(\mathcal{U}^{n,l}_{\dt}\sigma_k
\partial^3_xX \right).
\]
Therefore, we introduce the following process
 \[
  \mathcal{V}^{l}_{\dt} = \left\{ \begin{array}{lll} & \Id  \qquad &\text{for } 
 l = -1,0\\[0.25cm]
& T_{\dt,0}U^{-1}_{\dt,1} T^{-1}_{\dt,1}\qquad &\text{for }   l =1 \\[0.25cm]
			& T_{\dt,0}U^{-1}_{\dt,1} \cdots U^{-1}_{\dt,l-1}
T^{-1}_{\dt,l-1}
			\qquad &\text{for }   l \in \llbracket 2, n-1
\rrbracket,
                     \end{array}\right.
\] 
and separating the adapted part from the non-adapted part, we write 
\[\mathcal{U}^{n,l}_{\dt} = \mathcal{U}^{n,1}_{\dt}\left(
\mathcal{V}^{l-1}_{\dt} 
-\mathcal{V}^{l-2}_{\dt}+\mathcal{V}^{l-2}_{\dt} \right), 
\qquad \forall \ 1 \leqslant l \leqslant n. \] 
Now, using the unitarity property of $\mathcal{U}^{n,1}_{\dt} $ in
$\espace{L}^2$, 
we may write, for $q=1,2$, 
\be
 \lefteqn{\espace{E}\left( \max_{n \in \llbracket1,
N\rrbracket}\norm{\sum_{l=1}^n  \mathcal{U}^{n,l}_{\dt}  \epsilon_q^{l-1}  
}{\espace{L}^2}^{2p} \right)}\label{boundglobal}\\
&& \leqslant  C_p\espace{E}\left(  \max_{n \in \llbracket1,
N\rrbracket}\norm{\sum_{l=1}^n  \left( \mathcal{V}^{l-1}_{\dt}
-\mathcal{V}^{l-2}_{\dt}  \right) \epsilon_q^{l-1}   }{\espace{L}^2}^{2p}
\right)+
C_p\espace{E}\left( \max_{n \in \llbracket1, N\rrbracket}\norm{\sum_{l=1}^n 
\mathcal{V}^{l-2}_{\dt}  \epsilon_q^{l-1}   }{\espace{L}^2}^{2p} \right). \nonumber
\ee
Since $\mathcal{V}^{l-2}_{\dt}$ is $\mathcal{F}_{t_{l-1}}$ measurable, we are
allowed to use the Burkholder-Davis-Gundy inequality to estimate the second
term. The next Lemma, whose proof is postponed to section \ref{app}, gives
useful estimates to bound \eqref{boundglobal}.
\begin{lemma}\label{errorestimate}
For all $\left( f^l \right)_{l \in \llbracket 1, N \rrbracket} \in \left(
\espace{H}^1\left( \espace{R}\right)\right)^N  $ and for all $p>1$, there exists
a positive constant $C(\gamma, T, p)$, independent of $N$, such that
\begin{align}
\espace{E}\left(\max_{n  \in  \llbracket 1, N \rrbracket} \norm{\sum_{l=1}^n  
\left( \mathcal{V}^{l-1}_{\dt} - \mathcal{V}^{l-2}_{\dt}\right) f^{l}
}{\espace{L}^2}^{2p} \right) \leqslant C(\gamma, T, p) N^p \espace{E}\left(
\max_{n \in  \llbracket 1, N
\rrbracket}\norm{f^n}{\espace{H}^1}^{4p}\right)^{1/2}. \label{borne1}
\end{align}
Moreover, if for any $l \in \llbracket 1, N \rrbracket$, $f^{l}  = 
\epsilon_q^{l-1}, q=1,2$, then there exist two positive constants $C_1$ and $C_2$,
  independent of $N$, such that
\begin{align}
\espace{E}\left(\max_{n  \in  \llbracket 1, N \rrbracket} \norm{\sum_{l=1}^n  
\left( \mathcal{V}^{l-1}_{\dt} - \mathcal{V}^{l-2}_{\dt}\right) f^{l}
}{\espace{L}^2}^{2p} \right) \leqslant C_1(\gamma, T, p,q,
\norm{X_0}{\espace{H}^5})\dt^p. \label{borne2}
\end{align}
and 
\begin{align}
\espace{E}\left(\max_{n  \in  \llbracket 1, N \rrbracket} \norm{\sum_{l=1}^n   
\mathcal{V}^{l-2}_{\dt}  f^{l} }{\espace{L}^2}^{2p} \right) \leqslant
C_2(\gamma, T, p,q, \norm{X_0}{\espace{H}^4})\dt^p.\label{borne3}
\end{align}
\end{lemma} 
An estimate on \eqref{boundglobal} is easily obtained using \eqref{borne2} and \eqref{borne3} 
and we conclude the proof of Proposition \ref{linearstabilite}.
\end{proof}
\section{Probability and almost sure order for the Crank-Nicolson scheme
\eqref{CNApprox}}\label{AS}
This section is organized in two parts. In a first part, we will use, as is classical,
a cut-off argument on the nonlinear term which is not Lipschitz. We first define
a cut-off scheme, as an approximation of a continuous cut-off equation, 
and prove existence and uniqueness of a
global solution to this scheme.  The cut-off we
consider here for the scheme is of the same form as the one considered in
\cite{[bouardschema], [bouardorder]}. We also prove that the strong mean-square
rate of convergence of this approximation to the continuous cut-off equation is
$1/2$. This estimate is important in order to remove the cut-off.
In a second part, we construct a discrete solution to the Crank
Nicolson scheme \eqref{CNApprox} and define a discrete blow-up time. Using the
time order for the cut off scheme, we obtain a probability order and a.s. order
for the discrete scheme \eqref{CNApprox}, as is done in \cite{[bouardorder],
[printems]}.
\subsection{The lipschitz case}
 Let us denote by $U(t,s), t \geqslant s, t,s \in \espace{R}_+$ the random
unitary propagator defined as the unique solution of the linear equation 
\cite{[gazeau], [gazeauPHD]}
\[
idX(t)+  \pds{2}{x}{X(t)}dt+i\sqrt{\gamma}\sum_{k=1}^3
\sigma_k\pds{}{x}{X(t)}\circ dW_k(t)=0, \quad t \geqslant 0, x \in \espace{R}.
\]
Then, Equation \eqref{stochasticmanakov} with initial condition $X_0=v$,  can
be written in its mild form
 \begin{equation}
  X(t)=U(t,0)v+i \int_0^t U(t,s)\nl{}{X(s)}ds.
\end{equation}
We introduce a cut-off function $\Theta \in
C_c^{\infty}\left(\espace{R}\right)$, $\Theta \geqslant 0$ satisfying
$\Theta\left(x\right)=1$ for $x \in [0,1]$ and $\Theta\left(x\right)=0$ for $x
\geqslant 2$. We then define $\Theta_{R}\left(.\right)=\Theta\left(
\norm{.}{\espace{H}^1}/R\right)$ for any $R \in \espace{N}^*$. We set $ G\left(
\vecX{}{R}(s)\right) = \Theta^2_{R}\left( \vecX{}{R}(s)\right)\nl{}{
\vecX{}{R}(s)}$ and introduce the cut-off equation
\begin{equation}\label{duhamellimitetronquee}
  \vecX{}{R}(t)=U(t,0)v+i \int_0^t U(t,s)G\left( \vecX{}{R}(s)\right) ds,
\end{equation}
 which is the mild formulation of the equation
\begin{equation}\label{stochasticmanakovtronquee}
id\vecX{}{R}(t)+\left(\pds{2}{x}{\vecX{}{R}(t)} + G\left(
\vecX{}{R}(t)\right)\right) dt+i\sqrt{\gamma}\sum_{k=1}^3
\sigma_k\pds{}{x}{\vecX{}{R}(t)}\circ dW_k(t)=0.
\end{equation}
\subsubsection{Existence of a discrete solution}
Let us consider a semidiscrete scheme of equation
\eqref{stochasticmanakovtronquee}
\begin{equation}\label{CNtronquee}
\vecX{n+1}{R}= \vecX{n}{R} - H_{\dt,n}\vecX{n+1/2}{R} +i\dt G\left(\vecX{n
}{R},\vecX{n+1}{R} \right) 
\end{equation}
where   $G\left(\vecX{n }{R},\vecX{n+1}{R} \right)=\Theta_{\vecX{}{R}}^{n,n+1}
F\left(\vecX{n }{R},\vecX{n+1}{R} \right)$ and $\Theta_{\vecX{}{R}}^{n,n+1}    =
\Theta_{R}\left(\vecX{n}{R}\right)\Theta_{R}\left(\vecX{n+1}{R}\right)$. Such a
cut-off is used so that the discretization of the nonlinear term is consistent
with the continuous equation \eqref{stochasticmanakovtronquee}. Recall that the
nonlinear function $F$ is given by
\[
 F\left(\vecX{n }{R},\vecX{n+1 }{R}\right) =\frac{1}{2} \left(\abs{\vecX{n
}{R}}^2 +\abs{\vecX{n+1 }{R}}^2 \right)\vecX{n+1/2}{R}.
\]
 Now, we state in the next Proposition an existence and convergence result for
the scheme \eqref{CNtronquee}. This will be useful to define a solution, up to
the blow-up time, for \eqref{CNApprox} and a rate of convergence in a sense that
should be specified.
\begin{proposition}\label{existenceCN}
Let $\vecX{}{0} \in \espace{H}^1$ and $\dt >0$ fixed. Then there exists a unique adapted discrete
solution $\vecX{N}{R} = \left( \vecX{n}{R} \right)_{n = 0, \cdots, N_T}$ to
\eqref{CNtronquee} that belongs to $L^{\infty}\left(0, T; \espace{H}^1\right)$.
Furthermore for any $n \in \espace{N}$ such that $n  \leqslant N_T$, the
$\espace{L}^2$ norm is almost surely preserved i.e
$\norm{\vecX{n}{R}}{\espace{L}^2}=\norm{\vecX{ }{0}}{\espace{L}^2}$. 
\end{proposition}
To prove this result, we will use the next Lemma whose proof relies on the same
arguments as in \cite{[bouard], [gazeau]}.
\begin{lemma}\label{Lipschitz}
The function $G$ is a globally lipschitz continuous function from
$L^{\infty}\left(0, T; \espace{H}^1\times\espace{H}^1 \right)$ into
$L^{\infty}\left(0, T; \espace{H}^1\right)$ i.e. there exists a positive
constant $C$ independent of $N$ such that for any $\vecY{N}{R}$ and
$\vecX{N}{R}$ belonging to $L^{\infty}\left(0, T; \espace{H}^1\right)$ 
 \[
 \underset{n\dt \leqslant T}{\sup_{n \in \espace{N}^*}}
\norm{G\left(\vecX{n-1}{R},\vecX{n}{R}\right) -
G\left(\vecY{n-1}{R},\vecY{n}{R}\right)}{  \espace{H}^1  }\leqslant
CR^{2}\underset{n\dt \leqslant T}{\sup_{n \in
\espace{N}^*}}\norm{\vecX{n}{R}-\vecY{n}{R}}{ \espace{H}^1 } \quad \text{almost
surely.}
 \]
\end{lemma}
\begin{proof}[Proof of Proposition \ref{existenceCN}]
Assume that $\vecX{}{0} \in \espace{H}^1$, the integral formulation of the
cut-off scheme \eqref{CNtronquee} is then given by
\begin{equation}\label{discreteduhamelCN}
\vecX{n}{R} =  \mathcal{U}_{\dt}^{n,0}\vecX{}{0}  +  i\dt 
\sum_{l=1}^{n}\mathcal{U}_{\dt}^{n,l}G\left(\vecX{l-1}{R},\vecX{l}{R}\right),
\end{equation}
where $\mathcal{U}_{\dt}^{n,l}$ is the discrete random propagator solution of
the linear equation \eqref{ManakovPMDlimitelineairesemidiscret}. The proof
easily follows from the Lipschitz property of $G$. Moreover since 
$\mathcal{U}_{\dt}^{n,0}$ is an isometry in $\espace{L}^2$, the conservation of
the $\espace{L}^2$ norm follows taking the scalar product in $\espace{L}^2$ of
Equation \eqref{stochasticmanakovtronquee} with
$\left(\conjvecX{n+1/2}{R}\right)^t$.
\end{proof}
\subsubsection{Strong order of convergence}
Let us set $e^n_R = \vecX{n}{R}- \vecXexact{n}{R}$, where $\vecX{n}{R}$ is the
solution of \eqref{CNtronquee} and $\vecXexact{n}{R}$ is the solution of
\eqref{stochasticmanakovtronquee} evaluated at time $t_n$. The next result,
whose proof is postponed to Section \ref{app}, is crucial to obtain that the
strong order of convergence is $1/2$.
\begin{proposition}\label{lipschitzrate}
Let $X_0 \in \espace{H}^6$. For any $T \geqslant 0$ and $p \geqslant 1$, there
exists a positive constant $C$, depending on $R, T$ and $p$, and the
$\espace{H}^6$ norm of the initial data, such that
 \begin{align*}
  \espace{E}&\left( \max_{n = 0, \cdots, N}\norm{ \sum_{l=1}^{n} 
\int_{t_{l-1}}^{t_{l}} U\left(t_{n},s \right)G\left( \vecX{}{R}(s)\right) -
\mathcal{U}_{\dt}^{n,l}  G\left(\vecX{l-1}{R},\vecX{l}{R}\right)  ds
}{\espace{H}^1}^{2p} \right) \\ & \leqslant  C(R,T,p, \gamma,
\norm{X_0}{\espace{H}^6})\left[   \dt^{p} +    \espace{E}\left(  \max_{n \in
\llbracket 1 , N\rrbracket}\norm{ e^{ n}_{R}}{\espace{H}^1}^{2p} \right)\right],
 \end{align*}
where  the function $T \mapsto C(R,T,p, \gamma, \norm{X_0}{\espace{H}^6})$ is a
continuous function starting from zero.
\end{proposition}
As a consequence, we obtain 
\begin{proposition}\label{cvCN}
For any $T \geqslant 0$ and $p \geqslant 1$, there exists a positive constant
$C'$, depending on $R, T$ and $p$, and the $\espace{H}^6$ norm of the initial
data, such that
\begin{equation}\label{estimate_tronquee}
\espace{E}\left(\max_{n = 0, \cdots, N}\norm{e^n_R}{\espace{H}^1}^{2p}
\right)\leqslant C'(R, T, p , \gamma, \norm{X_0}{\espace{H}^6}) \dt^p.
\end{equation}
\end{proposition}
\begin{proof}[Proof of Proposition \ref{cvCN}]
Using the Duhamel formulation \eqref{duhamellimitetronquee} for the continuous
cut off equation and the discrete Duhamel equation \eqref{discreteduhamelCN},
and from Proposition \ref{linearstabilite} and \ref{lipschitzrate}, we obtain
for any $p \geqslant 1$
\a
 \espace{E}\left( \max_{n \in \llbracket 1 , N\rrbracket} \norm{ e^{
n}_{R}}{\espace{H}^1}^{2p}  \right)     \leqslant C(R,T,p, \gamma,
\norm{X_0}{\espace{H}^6})\left[   \dt^{p} +    \espace{E}\left(  \max_{n \in
\llbracket 1 , N\rrbracket}\norm{ e^{ n}_{R}}{\espace{H}^1}^{2p} \right)\right].
\b
Thus, for $T=T_1$ chosen sufficiently small so that $ C(T_1,R,p, \gamma,
\norm{X_0}{\espace{H}^6})<1$, we obtain 
\[
\espace{E}\left( \max_{n \in \llbracket 1 , N_{T_1}\rrbracket} \norm{e^{
n}_{R}}{\espace{H}^1}^{2p}  \right)  \leqslant \frac{C(T_1,R,p, \gamma,
\norm{X_0}{\espace{H}^6})}{1-C(T_1,R,p, \gamma, \norm{X_0}{\espace{H}^6})} 
\dt^{p}.
\] 
Iterating this process on the time intervals $[T_1, 2T_1]$ and up to the final
time $T$, we conclude that the scheme is of order $1/2$.
\end{proof}
\subsection{The non Lipschitz case}
In this section, we investigate the order in probability and the almost sure
order for the Crank-Nicolson scheme \eqref{CNApprox} as an approximation of
Equation \eqref{stochasticmanakov}. In order to define a discrete solution to
Equation \eqref{CNApprox}, let us define the random variable
\[
\vect{\tau}{R}{\dt} = \inf\left\lbrace  n\dt, \norm{\vecX{n-1}{R}}{\espace{H}^1}
\geqslant  R \quad \text{or} \quad \norm{\vecX{n}{R}}{\espace{H}^1} \geqslant  R
\right\rbrace, 
\]
which is a $\mathcal{F}_{n\dt}$ stopping time. It is then clear that $\left(
\vecX{n}{R}\right)_{n=0,\cdots, n_0-1} $ satisfy the scheme \eqref{CNApprox}
provided that $n_0\dt <  \vect{\tau}{R}{\dt} $. However, we do not know if a
solution $\vecX{n+1}{N}$ to \eqref{CNApprox} exists and is unique. We cannot
proceed as in the continuous case defining the blow-up time as the limit of
$\vect{\tau}{R}{\dt}$ when $R$ goes to infinity because the time step $\dt$
depends on the cut-off radius $R$ as it is seen in Proposition
\ref{existenceCN}. The next Lemma gives a sufficient condition on the time step
$\dt$ to extend the solution to $n+1$ \cite{[bouardorder]}. 
\begin{lemma}\label{stabilityball}
There exists a constant $C_2$ such that for any $\dt > 0$ and $R_0$ satisfying
$\dt \leqslant C_2R_0^{-2}$ and $n\dt \leqslant \tau^{R_0}_{\dt}$, there exists
a unique adapted solution $Z^{n+1}$ of
\begin{equation}\label{sol_t+1}
Z^{n+1} = U_{\dt,n}\vecX{n}{}+ i \dt T_{\dt,n}^{-1}F\left(\vecX{n }{},
Z^{n+1}\right) 
\end{equation}
such that $\norm{Z^{n+1}}{\espace{H}^1}\leqslant 4R_0$, provided
$\norm{\vecX{n}{}}{\espace{H}^1}\leqslant R_0$.
\end{lemma}
Following the approach of \cite{[bouardorder]}, we now define a new process
$Y^{n+1}_R$, solution of the truncated scheme \eqref{CNtronquee} with
$\vecX{n}{R} = \vecX{n}{}$, and we define the random variable
\[
 R_{n+1} = \min \left\lbrace R \in \espace{N}, \norm{Y^{n+1}_R}{\espace{H}^1}
\leqslant R \right\rbrace.
\]
 Fix any deterministic function $\vecX{ }{\dt, \infty}$ such that $\norm{\vecX{
}{\dt, \infty}}{\espace{H}^1} =4R_0$. Thus, for $\dt \leqslant C_2R_0^{-2}$, we
can define a solution of Equation \eqref{CNApprox} as follows
\be\label{constructsol}
\vecX{n+1}{} = \left\{ \begin{array}{ll}
Z^{n+1} \qquad \text{if} \quad  \norm{\vecX{n}{}}{\espace{H}^1}\leqslant R_0\\
 Y^{n+1}_{R_{n+1}} \qquad \text{if} \quad  \norm{\vecX{n}{}}{\espace{H}^1}> R_0
\text{ and }   R_{n+1} < +\infty  \text{ and }  \vecX{n}{} \neq \vecX{ }{\dt,
\infty}\\
\vecX{ }{\dt, \infty} \qquad \text{otherwise.}
\end{array}\right.
\ee
Finally, let $\tau^*_{\dt}$ be the discrete stopping time such that
$\tau^*_{\dt} = n_0\dt$ and $n_0$ is the first integer such that $\vecX{n}{} =
\vecX{ }{\dt, \infty}$. In this way, we define a solution to \eqref{CNApprox} up
to time $\tau^*_{\dt}$. The proof in \cite{[bouardorder]} can be adapted
straightforwardly to obtain the convergence in probability stated in Theorem
\ref{errorNL}. Note that from the almost sure convergence, we get, for any
stopping time $\tau < \tau^*$ a.s, $\lim_{\dt \to 0} \espace{P}\left(
\tau^*_{\dt} <\tau\right)  =0$.
Moreover, using the Fatou Lemma and the lower semicontinuity of the
characteristic function $\mathds{1}_{\tau^*_{\dt} <\tau}$, we obtain
$\espace{P}\left(\liminf_{\dt \to 0} \tau^*_{\dt} \geqslant \tau^* \right) =1$.

\section{Numerical almost sure error analysis}\label{numsim}
In this section, we study numerically the almost sure order of convergence of
the Crank Nicolson scheme \eqref{CNApprox} and with the aim of recovering the
theoretical result of the previous analysis. We consider finite-difference
approximation to simulate the $\espace{C}^2$ valued solution $X=(X_1, X_2)$ of
the stochastic Manakov system \eqref{stochasticmanakov}. We define a constant
$a>0$ and a final time $T>0$. The time step is $\Delta t = \frac{T}{N}>0$ and
the space step is given by $\Delta x=\frac{2a}{M+1} >0$. The grid is assumed to
be homogeneous $(t_n,x_j)=(n\Delta t, j\Delta x)$ for $n\in\left\{0,\ldots,
N\right\}$ and $j\in\left\{0,\ldots, M+1\right\}$. The computational domain
$[-a,a]$ is taken sufficiently large to avoid numerical reflections and we
consider  homogeneous Dirichlet boundary conditions. We denote $r = \dt/
(\dx)^2$ and the solution $X=(X_1, X_2)$ of Equation \eqref{stochasticmanakov},
evaluated at $(t_n,x_j)$, is approximated by $\vecX{n}{j}=\left( \vecX{n}{1,j},
\vecX{n}{2,j}\right) $. We choose a centered discretization due to the random
group velocity which does not have a well defined sign. The fully discrete
Crank-Nicolson scheme is given by
 \begin{align}\label{CNApproxfull}
 i\left( \vecX{n+1}{j} -  \vecX{n}{j} \right) & + r \Delta \vecX{n+1/2}{j} 
   + i\frac{\sqrt{\gamma r}}{2}\sum_{k=1}^3 \sigma_k \nabla \vecX{n+1/2}{j}
\chi_k^n  \\& 
  +\dz \frac{1}{2} \left(\abs{\vecX{n }{j}}^2 +\abs{\vecX{n+1 }{j}}^2
\right)\vecX{n+1/2}{j} =0 \nonumber,
 \end{align}
where 
\bearrays
\Delta \vecX{n+1/2}{j} = \vecX{n+1/2}{j-1}-2\vecX{n+1/2}{j}+\vecX{n+1/2}{j+1}\\
\nabla \vecX{n+1/2}{j} = \vecX{n+1/2}{j+1}-\vecX{n+1/2}{j-1}.
\eearrays
We consider soliton solutions of the deterministic Manakov equation as initial
input, that are  of the form  \cite{[Hasegawa_soliton_exp]}
\begin{align}\label{Manakov_soliton}
X(t,x)=  \begin{pmatrix}
\cos \Theta/2 \exp(i\phi_1)\\
\sin \Theta/2 \exp(i\phi_2)
\end{pmatrix} \eta  \text{sech} \eta(x-\tau(t))e^{-i k(x-\tau(t))+i\alpha(t)}.
\end{align} 
Here, the polarization angle $\Theta$, the phases $ \phi_1, \phi_2$, the
amplitude $\eta$ and the group velocity $-k$ are arbitrary constants and the
position $\tau$ and $\alpha$  are given by $\tau(t) = \tau_0 -kt$   and
$\alpha(t)= \alpha_0 +\frac{1}{2}\left(\eta^2 + k^2 \right) t$. We also define
the relative errors in the $\espace{L}^2$ and $\espace{L}^{\infty}$ norms
between the exact solution $\vecXexact{n}{}$, evaluated at time $t_n$, and the
approximated solution $\vecX{n}{}$ 
\begin{align}\label{relativeerror}
\text{err}_p^n  = \frac{\norm{\vecXexact{n}{} - \vecX{n}{}}{\espace{L}^p}}{
\norm{X_0}{\espace{L}^p}},  \qquad  p= \left\lbrace 2, \infty \right\rbrace.
\end{align}
The Stochastic Manakov equation possesses one invariant, which corresponds to
the mass. A discrete version of this quantity is given by
\begin{align}
 \norm{ \vecX{n}{}}{\espace{L}^2}^2 &= \dx \sum_{j=0}^{M+1} \left(
\abs{\vecX{n}{1,j}}^2 + \abs{\vecX{n}{2,j}}^2 \right).\label{discretemass}
\end{align}
To measure the ability of this scheme to preserve the mass, we introduce the
following error
\begin{align}
 \text{err}_{\espace{L}^2}^N &=
\max_{n \in \llbracket 1 , N\rrbracket} \abs{ \frac{\norm{\vecX{n}{
}}{\espace{L}^2}^2 - \norm{X_0}{\espace{L}^2}^2}{\norm{X_0}{\espace{L}^2}^2}
}.\label{errorinvnorm}
\end{align}
The set of parameters used for the simulations are given in the following Table
\ref{parameters_as}. 
\begin{table}[h]
\begin{center}
\begin{tabular}{|l  |c |}
\hline
& Almost-sure order    \\
\hline
\hline
 \multirow{1}{23mm}{Soliton} & 
$\phi_1= \phi_2 = k = \tau =0, \ \Theta=-\pi/2,\  \eta=1/2, \ \alpha_0=\pi, \  \gamma=0.1$ \\
\hline
\multirow{1}{23mm}{Discretization} & $a=30, \ M=20000, \ T=4, \ N_{\text{coarse}} = 40, \ N_{\text{fine}} = 2520$ \\
\hline
\end{tabular}
\end{center}
\caption{Set of parameters used to obtain the almost sure order. }
\label{parameters_as}
\end{table}
Since there is no explicit solution  for the stochastic Manakov equation, we
first compute an approximated solution $\vecX{n}{}$ of Equation
\eqref{stochasticmanakov} on a fine mesh $\dt= T/N_{\text{fine}}$, that we
compare to approximations of the same equation on coarser grids. A coarser grid,
in the $t$ variable, is twice as big as the previous one. The Brownian path is
kept fixed for each approximation as well as the space step $\dx$.  Figure
\ref{ascurve} displays two convergence curves corresponding to the logarithm of
the relative errors \eqref{relativeerror}. The slopes of these curves are
compared to a curve with slope $1/2$. From Fig. \ref{ascurve}, we see that the
almost sure order of the Crank Nicolson scheme is $1/2$ in the $t$ variable, and
the result agrees with the theoretical analysis of the previous section. Table
\ref{asresult} displays the numerical approximation errors in the $\espace{L}^2$
and $\espace{L}^{\infty}$ norms together with the relative error  for the
conservation of the mass. For an Euler scheme based on the Itô formulation, the
$\espace{L}^2$ norm is not preserved and the numerical error is
$\text{err}_{\espace{L}^2}^N = 0.7364$.

\saut{0.5}

\begin{minipage}{.45\textwidth} 
 \centerline{\includegraphics[width=.8\columnwidth]{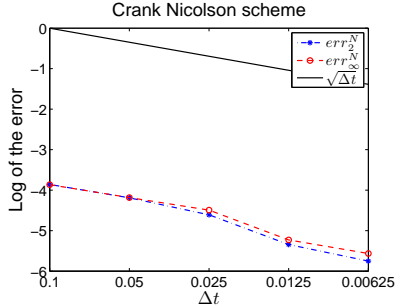} } 
\captionof{figure}{Plot of the log of the relative errors $\text{err}_2^N$ and
$\text{err}_{\infty}^N$ for the  scheme  \eqref{CNApproxfull}.}
\label{ascurve}
\end{minipage}  \qquad
\begin{minipage}{.45\textwidth}\centering
\begin{tabular}{|p{0.3\textwidth}|p{0.3\textwidth}|}
\hline
  & Crank-Nicolson    \\
\hline
\hline
$\text{err}_2^N$  & $3.178e^{-3}$   \\
\hline
$\text{err}_{\infty}^N$  & $3.83e^{-3}$    \\
\hline
$\text{err}_{\espace{L}^2}^N$  & $ 5.547e^{-11}$  \\
\hline
CPU time &  $ 251.87$s     \\
\hline
\end{tabular}
\captionof{table}{Numerical values of relative errors for $\Delta t = 0.00625$.}
\label{asresult}
\end{minipage}

\saut{0.5}

\noindent Different schemes may also be proposed to simulate the behaviour of
the solution of the stochastic Manakov equation \eqref{stochasticmanakov} : a
relaxation scheme and a Fourier split-step scheme. The fully discrete relaxation
scheme reads
\begin{eqnarray}\left\{ \begin{array}{ll} \label{RelaxApprox}
\vect{\Phi}{n+1/2}{j} = 2\abs{\vecX{n}{j}}^2 -  \vect{\Phi}{n-1/2}{j}\\[0.20cm]
 i\left( \vecX{n+1}{j} -  \vecX{n}{j} \right) +r  \Delta \vecX{n+1/2}{j}  
  + i\dfrac{\sqrt{\gamma r}}{2}\sum_{k=1}^3 \sigma_k \nabla \vecX{n+1/2}{j} 
\chi_k^n 
  \\ [0.20cm] \qquad   + \vect{\Phi}{n+1/2}{j} \vecX{n+1/2}{j}   \dz  
=0,
\end{array}\right.\end{eqnarray}
where $\vect{\Phi}{-1}{j}=\abs{\vecX{0}{j}}^2$.
 The stochastic Fourier split-step scheme is based on the decomposition of the
flow into two parts : one associated to the linear part of Equation
\eqref{stochasticmanakov} and the other to the nonlinear part. The scheme is
given by
\bearray\label{SplittingApprox}
 i\left( \widehat{Y}_{k}^{n+1} - \vecfX{n}{k}\right)  = m_k \left(
\widehat{Y}_{k}^{n+1} + \widehat{X}_{k}^{n}\right) \\[0.3cm]
  \vecX{n+1}{j} = \exp \left( i  \abs{Y^{n+1}_{j}}^2 \dz \right)
Y^{n+1}_{j},
 \eearray
where the Fourier multipliers $m_k$ are given by
\[
m_k=\left( \dfrac{\dz h_k^2}{2}   + \dfrac{\sqrt{\gamma \dz
}h_k}{2}\displaystyle{\sum_{l=1}^3 \sigma_l  \chi_l^n} \right)
\]
and
$\vecfX{n}{k}$ is the discrete Fourier transform of $\vecX{n}{j}$ and the
vector 
$ h$ contains the $M$ Fourier modes.
In this case, the matrices we have to invert for the linear step are block
diagonal. Consequently, this scheme is less time consuming than the relaxation
scheme and the Crank-Nicolson scheme. Figure \ref{ascurveOS1}  displays  the
almost sure error curves for these two schemes and they also seem to be of order
$1/2$. 
\begin{rmq}
In optics, spectral methods are very often used to solve the nonlinear
Schr\"odinger equation because the group 
associated to the free equation has an explicit and very simple form. The
random propagator, 
solution of the linear equation associated to 
\eqref{stochasticmanakov}, does not have an explicit formulation 
in Fourier space \cite{[gazeau], [gazeauPHD]}. Consequently a numerical approximation
of the linear equation is obtained resolving a linear system.
\end{rmq}
\begin{table}[htbp]
\begin{center}
\begin{tabular}{|c | c |c | c |c |}
\hline
  & $\text{err}_2^N$ & $\text{err}_{\infty}^N$ & $\text{err}_{\espace{L}^2}^N$  &  CPU time   \\
\hline
\hline
Fourier split-step  & $2.886e^{-3}$  & $3.455e^{-3}$ &$ 3.286e^{-14}$ & $ 180$s  \\
\hline
 relaxation &  $1.8e^{-3}$  &  $1.46e^{-3}$ & $3.957e^{-13}$ &  $121.53$s \\
\hline
\end{tabular}
\end{center}
\caption{Numerical values of relative errors for $\Delta t = 0.00625$.}
\end{table}

\begin{figure}
\includegraphics[width=.4\columnwidth]{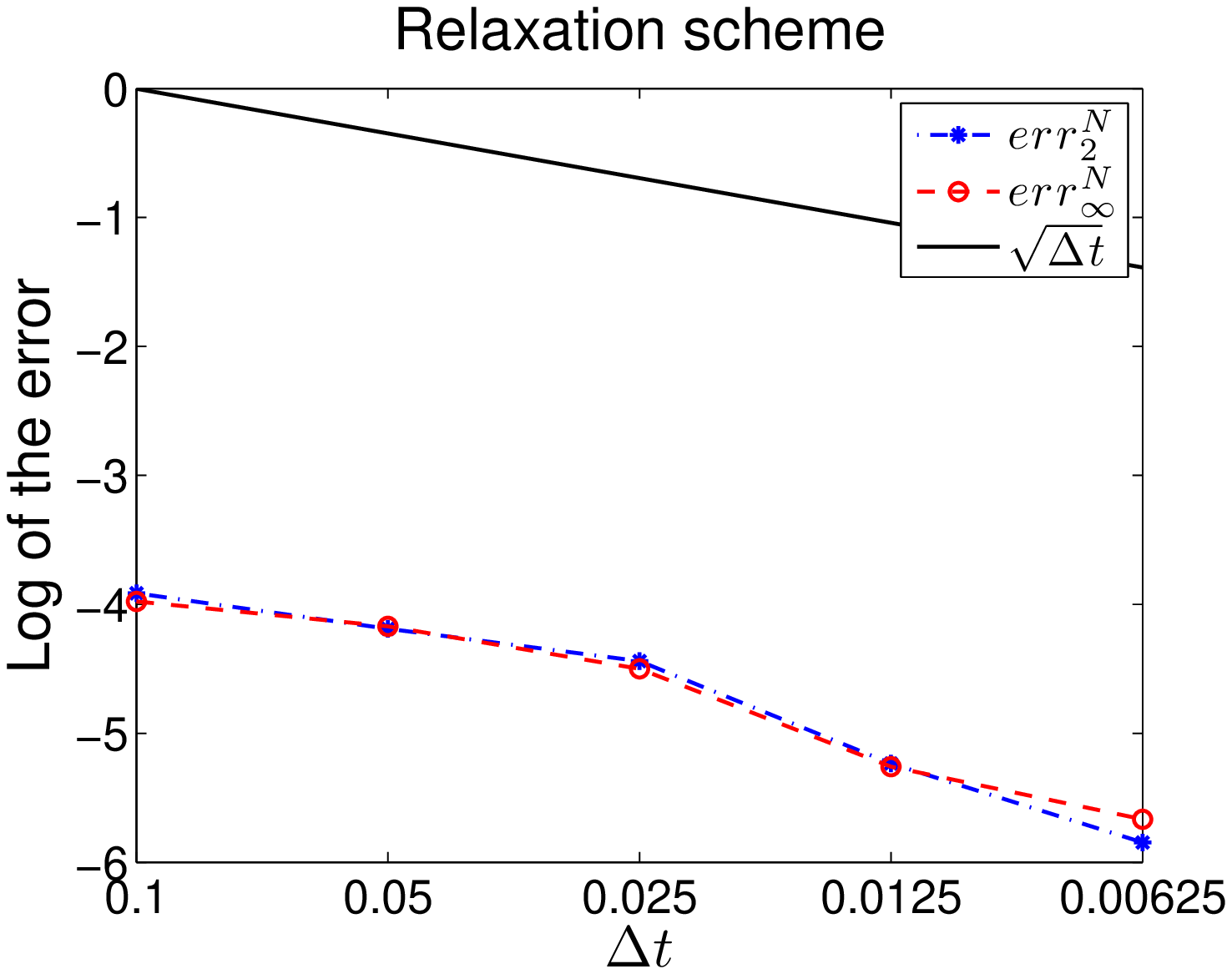}\hfill
\includegraphics[width=.4\columnwidth]{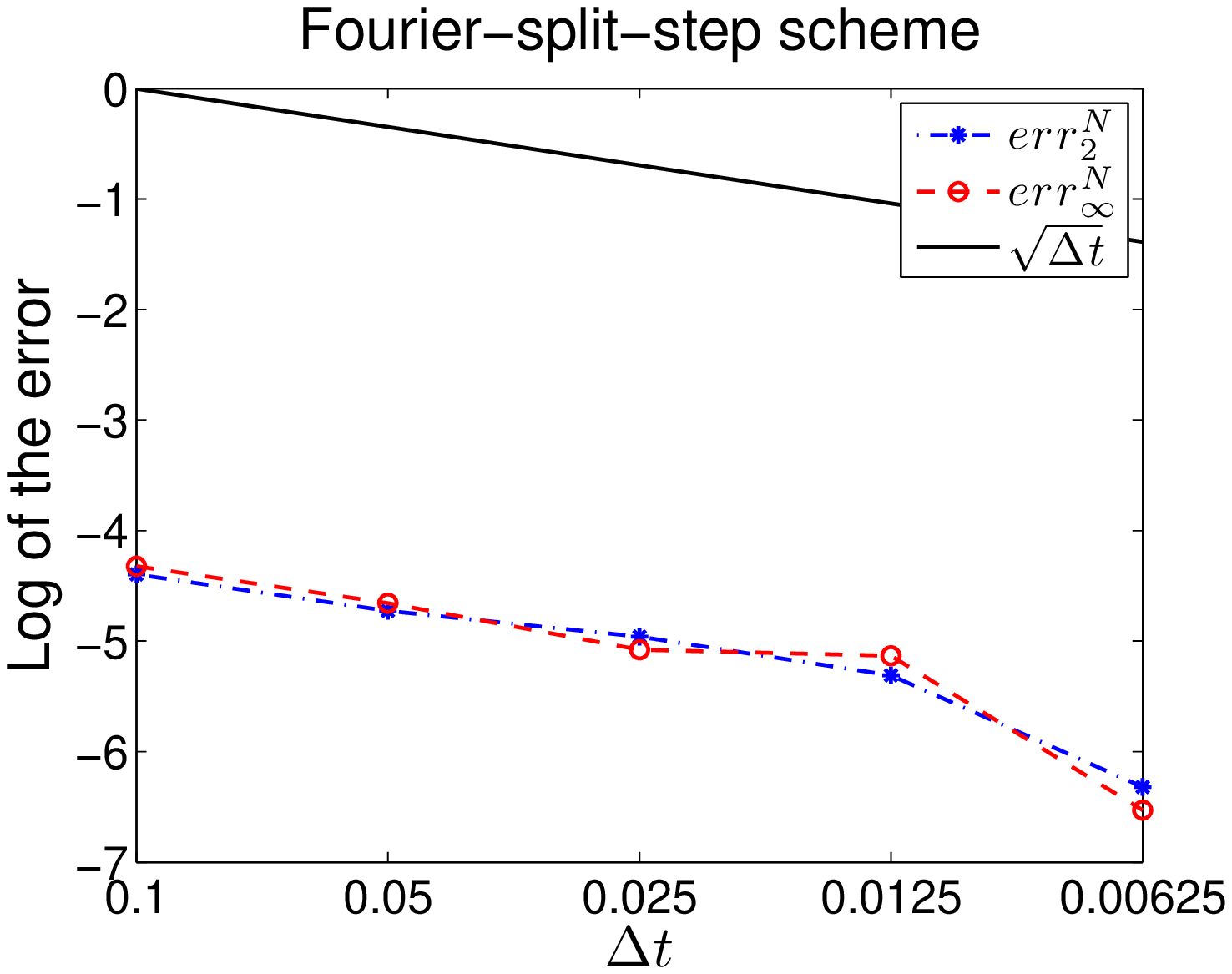}
\caption{Plot of the log of the relative errors $\text{err}_2^N$ and
$\text{err}_{\infty}^N$ respectively for the scheme \eqref{RelaxApprox} and
\eqref{SplittingApprox}}\label{ascurveOS1}
\end{figure}

\section{Proof of Lemma \ref{errorestimate} and Proposition
\ref{lipschitzrate}}\label{app}
\subsection{Proof of Lemma \ref{errorestimate}}
The proof of this lemma is divided into two parts. In a first step, we prove
inequality \eqref{borne1}. The second step consists in proving estimate
\eqref{borne3}; the same arguments are used to deduce the bound \eqref{borne2}
from \eqref{borne1}.
\paragraph{Proof of estimate \eqref{borne1}.}
We begin this proof with a lemma stating that $\mathcal{V}^{l}_{\dt}$ is almost
surely a bounded operator in $\espace{L}^2$ with a random continuity constant.
\begin{lemma}\label{continueL2}
The random matrix operator $\mathcal{V}^{l}_{\dt}$ is almost surely a bounded
operator in $\espace{L}^2$ and for any $l= 0 ,\cdots, n,$  
\[
 \norm{\mathcal{V}^{l}_{\dt} f }{\espace{L}^2} \leqslant C_{0,l}\left(\omega
\right)  \norm{f}{\espace{L}^2},
\]
such that for all $p \geqslant 1$, there exists a constant $C(p)$ independent of $n$, 
\[
\espace{E}\left(\max_{n  \in  \llbracket 1, N \rrbracket} C_{0,n}^{2p} \right) < C(p).
\]
\end{lemma}
\begin{proof}
By unitary property of the matrices $U_{\dt,l}$, for any $l= 0 ,\cdots, n$,
and applying Plancherel theorem and H\"older inequality, 
\begin{align*}
\norm{\mathcal{V}^{l}_{\dt} f }{\espace{L}^2} \leqslant \sup_{\xi \in
\espace{R}} M_{\dt,l}(\xi, \omega)  \norm{ f }{\espace{L}^2},
\end{align*}
where $M_{\dt,l}(\xi, \omega)= \normm{m_{\dt,0}\left(\xi \right)  }{\infty}
\normm{ m^{-1}_{\dt,l}\left(\xi \right) }{\infty}$, where $m_{\dt,l}$ is the
Fourier multiplier associated to the operator $T_{\dt,l}$. We claim that the
random variable $M_{\dt,l}(\xi, \omega)$ is almost surely bounded by a constant
$C_{0,l}(\omega)$, independent of $\xi$, that is integrable at any order.
Indeed,
\begin{align}\label{bornel2}
 & M_{\dt,l} (\xi, \omega)     \leqslant \frac{1}{4\abs{\det(l, \xi)}} \left(4
+4\dt\abs{\xi}^2 + \dt^{2}\abs{\xi}^4+
2\sqrt{\gamma\dt}\abs{\xi}\sum_{k=1}^3\left(
\abs{\chi_k^0}+\abs{\chi_k^l}\right) \right) \nonumber \\
&  +\frac{1}{4\abs{\det(l, \xi)}} \left(  
\sqrt{\gamma\dt}\dt\abs{\xi}^3\sum_{k=1}^3\left(
\abs{\chi_k^0}+\abs{\chi_k^l}\right)+\gamma\dt\abs{\xi}^2
\sum_{k=1}^3\abs{\chi_k^0}\sum_{k=1}^3\abs{\chi_k^l}    \right) ,  
\end{align}
where $\det(l, \xi)$ is the determinant of $ m_{\dt,l} $ and is given by
\[
 \det(l, \xi)= 1 + \frac{\gamma  \dt}{4} \sum_{k=1}^3\left( \chi_k^l \right)^2  
\abs{\xi}^2  - \frac{\dt^2}{4} \abs{\xi}^4-i\dt\abs{\xi}^2.
\]
Denoting $x=\dt^{1/2}\abs{\xi}$ and $y= \sum_{k=1}^3\left(\chi_k^l \right)^2 $,
we define the mapping $f$ from $  \espace{R}^2_+$ into $\espace{R}_+$
\[
f(x,y)= \left(1+\frac{\gamma x^2}{4}y - \frac{x^4}{4}
\right)^2 +x^4.
\]
It can be proved that
\begin{equation}
 f(x,y) \geqslant \left\{ \begin{array}{ll}
 \dfrac{1}{4}(1+x^4) & \text{if } x^2 \leqslant 4 \max\left( \dfrac{\gamma y}{4}, 1\right) \\[0.15cm]
  x^4  & \text{if } 4 \max\left( \dfrac{\gamma y}{4}, 1\right)  < x^2 \leqslant 16 \max\left( \dfrac{\gamma y}{4}, 1\right)\\[0.15cm]
  \dfrac{1}{32}x^8 +x^4  & \text{if } x^2 > 16 \max\left( \dfrac{\gamma y}{4}, 1\right).
\end{array}\right.
\end{equation}
Thus, there exists a positive constant $C$, such that for any $y \in
\espace{R}_+$ and any $\xi \in \espace{R}$
\begin{align*}
    M_{\dt,l} (\xi, \omega)    
<     \ C&\left( 1+\max\left( \frac{\gamma y}{4}, 1\right) +\sqrt{\gamma}\sum_{k=1}^3\left(
\abs{\chi_k^0}+\abs{\chi_k^l}\right)\left(1 + \max\left( \frac{\gamma y}{4}, 1\right)^{1/2} \right) 
\right.  \\& \left. 
+ \gamma \sum_{k=1}^3\abs{\chi_k^0}\sum_{k=1}^3\abs{\chi_k^l} \right) .
\end{align*}
 Therefore,  $M_{\dt,l} (\xi, \omega) $ is uniformly
bounded in $\xi$ by a polynomial function of $y$, $\abs{\chi_k^0}$ and
$\abs{\chi_k^l}$. Applying the Cauchy-Schwarz inequality and the Burkholder-Davis-Gundy
inequality, we obtain that $\espace{E}\left(\max_{n  \in  \llbracket
1, N \rrbracket}M_{\dt,n} (\xi)^{2p}\right)$ is bounded by a constant independent of $n$.
\end{proof}
We now state a Lemma giving an estimate of the local error between the unbounded
random operator $T_{\dt,n-1}U^{-1}_{\dt,n}T^{-1}_{\dt,n}$ and the identity
mapping. This Lemma will be used to prove inequality \eqref{borne1}.
\begin{lemma}\label{localerror}
For any $n \in \espace{N}$, there exists a positive random constant $C_{n-1,n}(
\omega) < +\infty$ a.s. belonging to $L^{2p}\left(\Omega\right), p \geqslant 1
$, such that for any $f \in \espace{H}^{1}$
\[
\norm{\left[T_{\dt,n-1}U^{-1}_{\dt,n}T^{-1}_{\dt,n} - \Id \right] f }{
\espace{L}^2} \leqslant C_{n-1,n}( \omega)\sqrt{\dt} \norm{f}{\espace{H}^1} 
\qquad a.s.
\]
Moreover, for all $p \geqslant 1$, there exists a constant $C(p)$ independent of $n$, \[
\espace{E}\left(\max_{n  \in  \llbracket 1, N \rrbracket} C_{n-1,n}^{2p} \right) < C(p). 
\]
\end{lemma}
\begin{proof}
From the proof of Lemma \ref{continueL2}, we easily deduce that there exists a random
variable $C_{n-1,n}(\omega)$, integrable at any order, such that
\[
 \sup_{\xi \in \espace{R}} \frac{1}{\left( 1+\abs{\xi}^2\right)^{1/2}}\normm{
\widehat{T}_{\dt,n-1}\widehat{U}^{-1}_{\dt,n}\widehat{T}^{-1}_{\dt,n}-\Id}{
\infty} \leqslant C_{n-1,n}( \omega)  \sqrt{\dt},
\]
where $C_{n-1,n}( \omega)$ is a polynomial function of $\sum_{k=1}^3\left(
\chi_k^n \right)^2$,  $\abs{\chi_k^{n-1}}$ and $\abs{\chi_k^n}$. The Cauchy-Schwarz 
and the Burkholder-Davis-Gundy inequalities imply that 
$\espace{E}\left(\max_{n  \in  \llbracket 1, N \rrbracket} C_{n-1,n}^{2p} \right)$ is bounded by a 
constant independent of $n$.
\end{proof}
Now, we prove estimate \eqref{borne1}. For any $l \geqslant 2$
  \begin{equation*}
 \mathcal{V}^{l-1}_{\dt} - \mathcal{V}^{l-2}_{\dt}  =
\mathcal{V}^{l-2}_{\dt}\left[\left( \Id +\frac{1}{2}H_{\dt,l-2}\right) \left(
\Id -\frac{1}{2}H_{\dt,l-1}\right)^{-1}  - \Id \right].
 \end{equation*}
Therefore applying Lemma \ref{continueL2} and \ref{localerror}, and using the
Cauchy-Schwarz inequality, we deduce that
 \begin{align*}
 \espace{E}&\left(\max_{n  \in  \llbracket 1, N \rrbracket} \norm{\sum_{l=1}^n  
\left( \mathcal{V}^{l-1}_{\dt} - \mathcal{V}^{l-2}_{\dt}\right) f^l
}{\espace{L}^2}^{2p} \right) \\
&\leqslant N^{2p} \espace{E}\left(\max_{n  \in  \llbracket 1, N
\rrbracket}\left( C_{0,n-2}\right)^{2p}
\norm{\left[T_{\dt,n-2}U^{-1}_{\dt,n-1}T^{-1}_{\dt,n-1} - \Id \right] f^n
}{\espace{L}^2}^{2p} \right) \\&\leqslant 
 T^p N^p \espace{E}\left(\max_{n  \in  \llbracket 1, N \rrbracket}\left(
C_{0,n-2}\right)^{4p}\left( C_{n-2,n-1}\right)^{4p}\right)^{1/2} 
\espace{E}\left(\max_{n  \in  \llbracket 1, N
\rrbracket}\norm{f^n}{\espace{H}^1}^{4p} \right)^{1/2}.
 \end{align*}
 Thus by Lemma \ref{continueL2}, the following inequality holds
\begin{align*}
  \espace{E}\left(\max_{n  \in  \llbracket 1, N \rrbracket} \norm{\sum_{l=1}^n  
\left( \mathcal{V}^{l-1}_{\dt} - \mathcal{V}^{l-2}_{\dt}\right) f^l
}{\espace{L}^2}^{2p} \right) \leqslant C(\gamma,T,p) N^p \espace{E}\left(\max_{n
 \in  \llbracket 1, N \rrbracket}\norm{f^n}{\espace{H}^1}^{4p} \right)^{1/2}.
 \end{align*}
\paragraph{Proof of estimate \eqref{borne3} for $q=1$.}
Writing $\vecXexact{}{}(s)= \vecXexact{l-1}{} + \widetilde{e}^{ \, l-1}$, we
rewrite $\epsilon_{1,1}^{l-1}$, given in \eqref{firstterm1}, as follows
\a
\epsilon_{1,1}^{l-1} = i C_{\gamma} \left( W_{0,0}^{l-1,l}\left( \partial^4_x \,
\widetilde{e}^{ \, l-1} \right) - \frac{1}{2}W_{0}^{l-1,l}\left( \partial^4_x\,
\widetilde{e}^{ \, l-1} \right) \dt \right).
\b
We focus on the first term in the above expression, the other term being bounded
in a similar way. Using the Minkowski inequality, the contraction property of
$\mathcal{U}_{\dt}^{n,l}$ in $\espace{L}^2\left(\espace{R} \right) $ for every
$l \in \llbracket1, n\rrbracket$ and the conservation of the $\espace{H}^4$
norm, we get 
\be\label{bound11}
\lefteqn{\espace{E}\left(\max_{n \in \llbracket 1, N
\rrbracket}\norm{\sum_{l=1}^n  
\mathcal{U}^{n,l}_{\dt}W_{0,0}^{l-1,l}(\partial_x^4 \, \widetilde{e}^{ \,l-1}) 
}{\espace{L}^2}^{2p} \right) } \nonumber\\
&\leqslant& \espace{E}\left(\left( \sum_{l=1}^N \sup_{t_{l-1}\leqslant u
\leqslant t_l}\norm{  \partial_x^4 \, \widetilde{e}^{ \,l-1}(u)}{\espace{L}^2}
\frac{\dt^2}{2}\right)^{2p} \right) \nonumber\\
&\leqslant& C\left( \norm{X_0}{\espace{H}^4 }^{2p}
\right)T^{2p}\dt^{2p}\left(1+\gamma^p\right).
\ee
Let us notice that after integration by part, the  term $\epsilon_{1,2}^{l-1}$
whose expression is given in \eqref{firstterm2}, can be written as follows
\be
\epsilon_{1,2}^{l-1} &=&  -i\sqrt{\gamma} \sum_{k=1}^3\sigma_k\partial^3_x
\vecXexact{l-1}{} \left(\frac{1}{2}W_k^{l-1}(1) \dt - W_{0,k}^{l-1,l}(1)\right)
\label{firstterm2bis}  \\
 &&- i\sqrt{\gamma} \sum_{k=1}^3 \left( W_{0,k}^{l-1,l}\left(\sigma_k
\partial^3_x \, \widetilde{e}^{ \, l-1} \right) - \frac{1}{2}W_{k}^{l-1,l}\left(
\sigma_k\partial^3_x \, \widetilde{e}^{ \, l-1} \right)\dt \right).\nonumber
\ee
Since $\mathcal{V}^{l-2}_{\dt}\sigma_k\partial^3_x \vecXexact{l-1}{}$ is
$\mathcal{F}_{l-1}$ adapted, the next equality holds 
\[
\mathcal{V}^{l-2}_{\dt} \sigma_k\partial^3_x \vecXexact{l-1}{}
W_{0,k}^{l-1,l}\left(1\right) = W_{0,k}^{l-1,l}\left(
\mathcal{V}^{l-2}_{\dt}\sigma_k\partial^3_x \vecXexact{l-1}{}\right).
\]
In expression \eqref{firstterm2bis}, all the terms may be bounded using similar
arguments. So, we only do the computation for the above term. By orthogonality
of the increments of the three dimensional Brownian Motion, 
\[
 \espace{E}\left(W_{0,k}^{l-1,l}\left(\mathcal{V}^{l-2}_{\dt}
\sigma_k\partial^3_x \vecXexact{l-1}{}\right) W_{0,j}^{l'-1,l'}\left(
\mathcal{V}^{l'-2}_{\dt}\sigma_j\partial^3_x \vecXexact{l'-1}{}\right) \right) =
0  \qquad \text{if } l \neq l' \ \text{or } k \neq j.
\]
Hence, we obtain
\[
\ll \sum_{l=1}^n \sum_{k=1}^3 W_{0,k}^{l-1,l}\left(\mathcal{V}^{l-2}_{\dt}
\sigma_k\partial^3_x \vecXexact{l-1}{}\right) \gg = \sum_{l=1}^n \sum_{k=1}^3
\ll W_{0,k}^{l-1,l}\left( \mathcal{V}^{l-2}_{\dt}\sigma_k\partial^3_x
\vecXexact{l-1}{} \right) \gg,
\]
where $\ll . \gg$ denotes the quadratic variation process. Thanks to the
conservation of the $\espace{H}^m$ norms, the solution $\vecX{}{}$ of Equation
\eqref{linearmanakovlimite} has all its moments bounded in $\espace{H}^m$ and
the stochastic integral is a true martingale. Thus, applying the
Burkholder-Davis-Gundy inequality, Lemma \ref{continueL2} and Cauchy-Schwarz
inequality, yields
\begin{align}\label{bound12}
 \espace{E}&\left(\max_{n  \in  \llbracket 1, N \rrbracket}
\gamma^p\norm{\sum_{l=1}^n  \sum_{k=1}^3 W_{0,k}^{l-1,l}\left(
\mathcal{V}^{l-2}_{\dt}\sigma_k\partial^3_x \vecXexact{l-1}{}\right)  
}{\espace{L}^2}^{2p} \right)\nonumber\\
&\leqslant CT^p\gamma^p \dt^{2p}\espace{E}\left(\max_{n  \in  \llbracket 1, N
\rrbracket}\left[C_{0,n-2} \right]^{2p} \norm{\partial^3_x \vecXexact{n-1}{}
}{\espace{L}^2}^{2p} \right) \nonumber\\
&\leqslant CT^p\gamma^p \dt^{2p}\espace{E}\left(\max_{n  \in  \llbracket 1, N
\rrbracket}\left[ C_{0,n-2} \right]^{4p}\right)^{1/2}  \espace{E}\left(\max_{n 
\in  \llbracket 1, N \rrbracket}\norm{\partial^3_x \vecXexact{n-1}{}
}{\espace{L}^2}^{4p} \right)^{1/2}.
\end{align}
Hence, a bound follows from the conservation of the $\espace{H}^m$ norms and Lemma 
\ref{continueL2}. Collecting the
above estimates \eqref{bound11} and \eqref{bound12} leads to the bound \eqref{borne3} for $q=1$.
 \paragraph{Proof of estimate \eqref{borne3} for $q=2$.}
The first and second terms $\epsilon^{l-1}_{2,1}$ and $\epsilon^{l-1}_{2,2}$ in
\eqref{secondterm}  will give the order of convergence of the scheme. The third
one $\epsilon^{l-1}_{2,3}$ may be bounded similarly as in the previous step. To
bound $\epsilon^{l-1}_{2,1}$, we use again the Burkholder-Davis-Gundy
inequality, the independence of the increments of the Brownian Motion, Lemma
\ref{continueL2}, Cauchy-Schwarz inequality and Lemma \ref{strong_estimate_sol} 
\begin{align}
\espace{E}&\left(\max_{n \in \llbracket 1, N \rrbracket} \gamma^p
\norm{\sum_{l=1}^n   \sum_{k=1}^3    W_k^{l-1,l}\left(\mathcal{V}^{l-2}_{\dt}
\sigma_k \partial_x \widetilde{e}^{\, l-1} \right)}{\espace{L}^2}^{2p} \right) \nonumber\\
&\leqslant
C\gamma^p \espace{E}\left( \left(  \sum_{l=1}^N \sum_{k=1}^3 W_0^{l-1,l}\left(
\norm{\mathcal{V}^{l-2}_{\dt}\sigma_k \partial_x \widetilde{e}^{\,
l-1}}{\espace{L}^2}^{2} \right) \right)^p \right)\nonumber \\
&\leqslant   C(p,\gamma)\norm{X_0}{\espace{H}^2}^{2p} T^{p}\dt^{p}. \label{bound21}
\end{align}
We conclude the proof obtaining an estimate for $\epsilon^{l-1}_{2,2}$. Using
Equation \eqref{It\^oform} and Property \ref{pauli_matrices}, we obtain the
equality
\a
\epsilon_{2,2}^{l-1} &= & \frac{3 \gamma}{2}\pds{2}{x}{\vecXexact{l-1}{}} \dt
-\frac{ \gamma}{2}\sum_{k=1}^3  \pds{2}{x}{\vecXexact{l-1}{}}   \left(
W^{l-1}_k(1)\right)^2\\
&&+ \frac{3 \gamma}{2} W_0^{l-1,l}\left( \partial_x^2 \widetilde{e}^{\, l-1}
\right)  -\frac{ \gamma}{2}\sum_{j,k=1}^3 W_j^{l-1,l}\left(\sigma_j\sigma_k 
\partial_x^2 \widetilde{e}^{\, l-1}\right)  W^{l-1}_k(1).
\b
Moreover,
\[
 \ll \left( W_k^{l-1,l}(1)\right)^2 - \dt \gg =4\dt\left(
W_k^{l-1,l}(1)\right)^2 - 2\dt ^2 .
\]
Thus, applying the Burkholder-Davis-Gundy inequality, using the independence of
the increments of the Brownian Motion, applying Lemma \ref{continueL2}, using
the conservation of the $\espace{H}^m$ norms and the Cauchy-Schwarz inequality 
\begin{align}\label{bound22}
\espace{E}&\left(\max_{n \in \llbracket 1, N \rrbracket} \norm{\sum_{l=1}^n 
\mathcal{V}^{l-2}_{\dt} \left(\frac{3 \gamma}{2}\pds{2}{x}{\vecXexact{l-1}{}}
\dt -\frac{ \gamma}{2}\sum_{k=1}^3  \pds{2}{x}{\vecXexact{l-1}{}}   \left(
W^{l-1}_k(1)\right)^2\right) }{\espace{L}^2}^{2p} \right) \nonumber\\ &\leqslant 
C \gamma^{2p} \espace{E}\left(  \left( \sum_{l=1}^N \sum_{k=1}^3 
\norm{\mathcal{V}^{l-2}_{\dt}\pds{2}{x}{\vecXexact{l-1}{}} \left( 4\dt \left(
W_k^{l-1}(1)\right)^2 - 2\dt^2\right) }{\espace{L}^2} \right)^p \right) \nonumber
\\& \leqslant 
C\norm{X_0}{\espace{H}^2}^{2p} \gamma^{2p} N^{p-1}\sum_{l=1}^N\sum_{k=1}^3 \espace{E}\left(  
\left( C_{0,l-2}\right)^p  \abs{4\dt \left( W_k^{l-1}(1)\right)^2 - 2\dt^2}^p 
\right)\nonumber\\
&\leqslant  C\norm{X_0}{\espace{H}^2}^{2p}\gamma^{2p}T^p\dt^{p}.
\end{align}
The last term
$\epsilon^n_{2,3}$ in \eqref{secondterm} may be bounded similarly as
$\epsilon^n_{1,2}$. Estimate \eqref{borne3} for $q=2$ is obtained collecting bounds 
\eqref{bound21} and \eqref{bound22}.

\subsection{Proof of Proposition \ref{lipschitzrate}}
Before proving Proposition \ref{lipschitzrate}, let us state two useful Lemmas.
The first result gives uniform bounds for the solution $\vecX{}{R}$ of the
cut-off equation \eqref{stochasticmanakovtronquee}.
\begin{lemma}\label{boundinH1}
 Let $\vecX{}{0} \in \espace{H}^6$ and $\vecX{}{R}$ be the solution of
\eqref{stochasticmanakovtronquee}; then for all $T>0$ there exists a positive
constant $C_3\left(R,T ,\norm{X_0}{\espace{H}^6} \right)$, such that, a.s for
every $t$ in $[0,T]$,
\[
\norm{\vecX{}{R}(t)}{\espace{H}^6} \leqslant
C_3\left(R,T,\norm{X_0}{\espace{H}^6} \right).
\]
Moreover, the function $T \mapsto C_3\left(R,T,\norm{X_0}{\espace{H}^6} \right)$
is a continuous function from $\espace{R}_+$ to $\espace{R}_+$ and then is
bounded on every compact set of $\espace{R}_+$. We denote by
$\widetilde{C}_3\left(R,T ,\norm{X_0}{\espace{H}^6} \right)$ the positive
constant such that, a.s for every $t$ in $[0,T]$,
\[
\norm{F\left( \vecX{}{R}(t)\right) }{\espace{H}^6} \leqslant
\widetilde{C}_3\left(R,T,\norm{X_0}{\espace{H}^6} \right).
\]
\end{lemma}
\begin{proof}[Proof of Lemma \ref{boundinH1}]
The proof is similar to the proof of Lemma $4.1$ in \cite{[gazeau]}.
\end{proof}
Let us now denote by $\widetilde{e}^{ \, n}_{R}(s) $ the difference
$\widetilde{e}^{ \, n}_{R}(s) = \vecX{}{R}(s) -  \vecXexact{n}{R}$ for all $s
\in [t_n, t_{n+1}]$ and state an intermediate result which gives a local
estimate on $\widetilde{e}^{ \, n}_{R}(s)$.
\begin{lemma}\label{estimate_sol_tronquee}
 For any $p \geqslant 1$, if $\vecX{}{0} \in \espace{H}^{2}$ then there exists a
positive constant $C_4$, such that
\[ 
\espace{E}\left(\sup_{t_{l-1} \leqslant t \leqslant t_l}  \norm{\widetilde{e}^{
\, l-1}_{R}(t)}{\espace{H}^1}^{2p} \right) \leqslant C_4(R,T,p, \gamma,
\norm{\vecX{ }{0}}{\espace{H}^{2}}) \dt^{p}  \qquad \forall \ l = 1, \cdots,
N_T.
\]
Moreover $C_4 \equiv C_p(\gamma)C_3^{2p}(R,T, \norm{\vecX{
}{0}}{\espace{H}^{2}}) + C(R)\dt^p$, where $C_p(\gamma)$ is given in Lemma
\ref{strong_estimate_sol}, $C_3$ is given in Lemma \ref{boundinH1} and $C(R)$ is a positive constant depending only on
$R$.
\end{lemma}
\begin{proof}[Proof of Lemma \ref{estimate_sol_tronquee}]
This estimate is obtained using the Duhamel formulation
\eqref{duhamellimitetronquee}, writing $\vecX{}{R}(t)$ in terms of
$\vecXexact{l-1}{R}$, using Lemma  \ref{strong_estimate_sol} and \ref{boundinH1}
and because $G$ is globally Lipschitz.
\end{proof}
Let us now prove  Proposition \ref{lipschitzrate}.
\begin{proof}[Proof of Proposition \ref{lipschitzrate}]
We split the difference as follows
 \begin{align*}
 U\left(t_{n},s \right)G\left( \vecX{}{R}(s)\right) - \mathcal{U}_{\dt}^{n,l} 
G\left(\vecX{l-1}{R},\vecX{l}{R}\right)  =  A_1^{l-1,l} + A_2^{l-1,l}  +
A_3^{l-1,l}, 
 \end{align*}
 where
\bearray
& A_1^{l-1,l}=\displaystyle{ U(t_n,s)\left(
\Theta_{R}^2\left(\vecX{}{R}(s)\right)- \Theta_{\vecX{}{R}}^{l-1,l}
\right)F\left( \vecX{}{R}(s)\right)  }\\[0.5cm]
& A_2^{l-1,l}=\displaystyle{ \left( U\left(t_{n},s \right) - 
\mathcal{U}_{\dt}^{n,l}  \right) \Theta_{\vecX{}{R}}^{l-1,l} F\left(
\vecX{}{R}(s)\right)  }\\[0.5cm]
& A_3^{l-1,l}=\displaystyle{ 
 \mathcal{U}_{\dt}^{n,l}\Theta_{\vecX{}{R}}^{l-1,l}  \left(F\left(
\vecX{}{R}(s)\right) -  F\left(\vecX{l-1}{R},\vecX{l}{R}\right) \right) }.
\eearray
In order to obtain an estimate on the global error in $L^{2p}\left(\Omega\right)
$, we decompose the term $\vecX{}{R}(s) - \vecX{l}{R}$, appearing in
$A_1^{l-1,l}$ and $A_3^{l-1,l}$,  in two terms : $\widetilde{e}^{ \, l}_{R}(s)$ 
and $e_R^l$. The first term gives the contribution to the final order and the
second term may be handled by a fixed point procedure. Let us denote
$\Theta_{\vecX{}{R}}^{l} = \Theta\left( \norm{\vecX{l}{R}}{\espace{H}^1}/R
\right) $ for any $l=0, \cdots n$. Writing
\begin{align*}
\Theta_{R}^2\left(\vecX{}{R}(s)\right) - \Theta_{\vecX{}{R}}^{l-1,l} = \ &
\Theta_{R}\left(\vecX{}{R}(s)\right) \left( \Theta_{R}\left(\vecX{}{R}(s)\right)
 - \Theta_{\vecX{}{R}}^{l-1}\right)\\ & - \Theta_{\vecX{}{R}}^{l-1}\left(
\Theta_{R}\left(\vecX{}{R}(s)\right)  - \Theta_{\vecX{}{R}}^{l}\right) 
\end{align*}
and using the isometric property of the random propagator $U$, the boundedness of
$\Theta$ and $\Theta'$ and the mean value theorem we obtain the following bound
\a
 \lefteqn{\norm{\sum_{l=1}^n
\int_{t_{l-1}}^{t_l}\Theta_{R}\left(\vecX{}{R}(s)\right) \left(
\Theta_{R}\left(\vecX{}{R}(s)\right)-
\Theta_{\vecX{}{R}}^{l-1}\right)U(t_n,s)F\left(\vecX{}{R}(s)\right)ds
}{\espace{H}^1}^{2p} } \\
 &&\leqslant \left( \sum_{l=1}^n CR^3 \int_{t_{l-1}}^{t_l}
\frac{\norm{\Theta'}{L^{\infty}}}{R} \left( \norm{\widetilde{e}^{ \,
l-1}_{R}(s)}{\espace{H}^1} + \norm{e^{ \, l-1}_{R}}{\espace{H}^1}\right) ds
\right)^{2p}.
\b
By the same arguments, together with Lemma \ref{boundinH1} and
\ref{estimate_sol_tronquee}, we obtain  
 \a
 \espace{E}\left( \max_{n \in \llbracket 1 , N\rrbracket}\norm{\sum_{l=1}^n
A_1^{l-1,l}}{\espace{H}^1}^{2p} \right) \leqslant  C_5(R,T,p)\left[ C_4(R,T,p,
\gamma, \norm{\vecX{ }{0}}{\espace{H}^{2}}) \dt^p +  \espace{E}\left(  \max_{n
\in \llbracket 1 , N\rrbracket}\norm{e^{n}_{R}}{\espace{H}^1}^{2p}
\right)\right].
\b
where $C_5 \equiv C^{2p}R^{4p}T^{2p}.$
Now, we split the term $A_2^{l-1,l}$ further
\a
A_2^{l-1,l}
&=&  \int_{t_{l-1}}^{t_l}U\left(t_{n},s \right)\left(\Id - 
U\left(s,t_{l-1}\right)  \right) \Theta_{\vecX{}{R}}^{l-1,l} F\left(
\vecX{}{R}(s)\right)ds\\
&&+\int_{t_{l-1}}^{t_l}\left( U\left(t_{n},t_{l-1} \right) -
\mathcal{U}_{\dt}^{n,l}\right) \Theta_{\vecX{}{R}}^{l-1,l}F\left(
\vecX{}{R}(s)\right)ds\\
&=& A_{2,1}^{l-1,l}+A_{2,2}^{l-1,l}.
\b
The first term in the above equality can easily be estimated using again the
isometric property of the random propagator $U\left(t_{n},s \right)$ and Hölder
inequality, together with Lemma \ref{boundinH1} and \ref{strong_estimate_sol}, 
\begin{align*}
\espace{E} \left(\max_{n \in  \llbracket 1 , N\rrbracket}\norm{\sum_{l=1}^n
A_{2,1}^{l-1,l}}{\espace{H}^1}^{2p} \right) \leqslant C_6(R,T,p,\gamma ,
\norm{X_0}{\espace{H}^2})  \dt^p.
\end{align*}
On the contrary, the second term $A_{2,2}^{l-1,l}$ cannot be bounded directly
because we do not have an explicit representation (in Fourier space) of the
random propagator $U(t,s)$, $t,s \in \espace{R}_+,  t\geqslant s$, solution of
the linear equation \eqref{linearmanakovlimite}. Writing  
\[
\mathcal{U}_{\dt}^{n,l} = U_{\dt,n}\cdots U_{\dt,l-1}\left(\Id
-\frac{1}{2}H_{\dt,l-1} \right)^{-1},
\]
we split $A_{2,2}^{l-1,l}$ as follows
\a
A_{2,2}^{l-1,l} &=&  \int_{t_{l-1}}^{t_l}\left( U\left(t_{n},t_{l-1} \right) -
\mathcal{U}_{\dt}^{n,l}\right) \Theta_{\vecX{}{R}}^{l-1,l} \left( F\left(
\vecX{}{R}(s)\right) -   F\left( \vecXexact{l-1}{R} \right)\right)ds \\
&&+ \int_{t_{l-1}}^{t_l} \Theta_{\vecX{}{R}}^{l-1,l}\left( U\left(t_n, t_{l-1}
\right) - U_{\dt,n}\cdots U_{\dt,l-1} \right)F\left( \vecXexact{l-1}{R} \right)
ds\\
&&+  \int_{t_{l-1}}^{t_l} \Theta_{\vecX{}{R}}^{l-1,l}U_{\dt,n}\cdots
U_{\dt,l-1}\left(\Id -  \left(\Id -\frac{1}{2}H_{\dt,l-1} \right)^{-1} 
\right)F\left( \vecXexact{l-1}{R}\right) ds\\
&=& A_{2,2,1}^{l-1,l}+A_{2,2,2}^{l-1,l}+A_{2,2,3}^{l-1,l}.
\b
The first term $A_{2,2,1}^{l-1,l}$ is easily bounded thanks to the local
Lipschitz property of the nonlinear function $F$, the isometric property of both
$U\left(t_{n},t_{l-1} \right)$ and $\mathcal{U}_{\dt}^{n,l}$, the boundedness of
$\Theta$ and Lemma \ref{boundinH1}. This leads to 
\begin{align}\label{estimateA221}
\espace{E}\left(\max_{n \in  \llbracket 1 , N\rrbracket}\norm{\sum_{l=1}^n
A_{2,2,1}^{l-1,l}}{\espace{H}^1}^{2p} \right) \leqslant C_7(R,T,p, \gamma,
\norm{X_0}{\espace{H}^2})   \dt^p.
\end{align}
Let us now consider the second term $A_{2,2,2}^{l-1,l}$ that can be bounded
using the linear estimate \eqref{strongorder} obtained in Proposition
\ref{linearstabilite} together with Lemma \ref{boundinH1}. In this way,
\begin{equation}\label{estimateA222}
\espace{E} \left(\max_{n \in  \llbracket 1 , N\rrbracket}\norm{\sum_{l=1}^n
A_{2,2,2}^{l-1,l}}{\espace{H}^1}^{2p} \right) \leqslant C_8(R,T,p, \gamma,
\norm{X_0}{\espace{H}^6}) \dt^p.
\end{equation} 
An estimate on the last term $A_{2,2,3}^{l-1,l}$ is obtained thanks to the next
result, whose proof is identical to Lemma \ref{localerror}.
\begin{lemma}\label{localerrorNL}
For any $n \in \espace{N}$, there exists a positive random constant $C_{n}(
\omega) < +\infty$ a.s. belonging to $L^{2p}\left(\Omega\right), p \geqslant 1
$, whose moments are independent of $n$, such that for any $f \in
\espace{H}^{1}$
\[
\norm{\left[ \Id -  \left(\Id -\frac{1}{2}H_{\dt,n} \right)^{-1}  \right] f }{
\espace{L}^2} \leqslant C_{n}( \omega)\sqrt{\dt} \norm{f}{\espace{H}^1}  \qquad
a.s.
\]
Moreover for any $ p \geqslant 1$, there exists a constant $C(p)$ independent of $n$ such that
\[
 \espace{E}\left(\max_{n \in \llbracket 1 , N\rrbracket}C_{n}^{2p} \right)
< C(p).
\]
\end{lemma}
From this Lemma, we easily obtain a bound on the last term $A_{2,2,3}$.
Combining the above estimates \eqref{estimateA221}, \eqref{estimateA222}, we
obtain an estimate on  $A^{l-1,l}_{2,2}$
\begin{align*}
 \espace{E}\left( \max_{n \in \llbracket 1 , N\rrbracket}\norm{\sum_{l=1}^n
A^{l-1,l}_{2,2}}{\espace{H}^1}^{2p} \right)  \leqslant C_9(R,T,p, \gamma,
\norm{X_0}{\espace{H}^6})  \dt^p.
\end{align*}
Finally, we bound the last term $A_3^{l-1,l}$ splitting it as follows
\a
A_3^{l-1,l} &=& \int_{t_{l-1}}^{t_l}
\mathcal{U}_{\dt}^{n,l}\Theta_{\vecX{}{R}}^{l-1,l}  \left(F\left(
\vecX{}{R}(s)\right) -  F\left(\vecXexact{l-1}{R},\vecXexact{l-1}{R}\right)
\right)ds  \\&&+
\int_{t_{l-1}}^{t_l} \mathcal{U}_{\dt}^{n,l}\Theta_{\vecX{}{R}}^{l-1,l}  \left(
F\left(\vecXexact{l-1}{R},\vecXexact{l-1}{R}\right)- 
F\left(\vecXexact{l-1}{R},\vecXexact{l}{R}\right) \right)ds\\&&+
\int_{t_{l-1}}^{t_l} \mathcal{U}_{\dt}^{n,l}\Theta_{\vecX{}{R}}^{l-1,l} 
\left(F\left(\vecXexact{l-1}{R},\vecXexact{l}{R}\right)- 
F\left(\vecX{l-1}{R},\vecX{l}{R}\right) \right)ds \\
&=& A_{3,1}^{l-1,l}+A_{3,2}^{l-1,l}+A_{3,3}^{l-1,l}.
\b
Note that by Lemma \ref{Lipschitz}, the last term $A_{3,3}^{l-1,l}$ is easily
bounded as follows
\[
  \espace{E}\left( \max_{n \in \llbracket 1 , N\rrbracket}\norm{\sum_{l=1}^n
A^{l-1,l}_{3,3}}{\espace{H}^1}^{2p} \right)  \leqslant
C_{10}(R,T,p)\espace{E}\left(  \max_{n \in \llbracket 1 ,
N\rrbracket}\norm{e^{n}_{R}}{\espace{H}^1}^{2p} \right),
\]
where $C_{10} \equiv C^{2p} R^{4p} T^{2p}$. The first term $A_{3,1}^{l-1,l}$ is
bounded using
$F\left(\vecXexact{l-1}{R},\vecXexact{l-1}{R}\right)=F\left(\vecXexact{l-1}{R}
\right)$, Lemma \ref{boundinH1}, Hölder inequality and Lemma
\ref{estimate_sol_tronquee}. An estimate on the second term $A_{3,2}^{l-1,l}$
may be obtained using Lemma  \ref{boundinH1} and Lemma
\ref{estimate_sol_tronquee}.
\end{proof}

\section{Conclusion}
The evolution of the slowly varying envelopes 
driven by random polarization mode dispersion is described by the stochastic 
Manakov equation. We introduce three different schemes for this equation using a semi-implicit 
discretization of the Stratonovich integrals. We prove
that the CN scheme is of order $1/2$ and is conservative for
the discrete
$\espace{L}^2$ norm, contrarily to a scheme based on the Itô formulation. 
This method may be applied 
to other stochastic equations written in Stratonovich form and especially for equations
with conservation laws.

\def\polhk#1{\setbox0=\hbox{#1}{\ooalign{\hidewidth
  \lower1.5ex\hbox{`}\hidewidth\crcr\unhbox0}}} \def\cprime{$'$}

\end{document}